\documentclass{article}
\usepackage{amsfonts}
\usepackage{amsmath}
\usepackage[margin=1.15in]{geometry}

\newtheorem{theorem}{Theorem}

\newtheorem{definition}[theorem]{Definition}

\newtheorem{lemma}[theorem]{Lemma}
\newtheorem{notation}[theorem]{Notation}

\newtheorem{remark}[theorem]{Remark}

\newenvironment{proof}[1][Proof]{\noindent\textbf{#1.} }{\ \rule{0.5em}{0.5em}}
\input{tcilatex}
\begin{document}

\title{Dimension-free Euler estimates of rough differential equations}
\author{Youness Boutaib, Lajos Gergely Gyurk\'{o}, Terry Lyons, Danyu Yang
\and University of Oxford and Oxford-Man Institute}
\maketitle

\begin{abstract}
We extend the result in \cite{Davie A. M}, \cite{Friz and Victoir} and \cite%
{Friz and Victoir book}, and give a dimension-free Euler estimation of
solution of rough differential equations in term of the driving rough path.
In the meanwhile, we prove that, the solution of rough differential equation
is close to the exponential of a Lie series, with a concrete error bound.
\end{abstract}

\section{Introduction}

Suppose that $X$ is a continuous bounded variation path defined on some
interval $I$ and taking its values in a Banach space $\mathcal{V}$. We view
this path as a stream of information and allow that it is highly oscillatory
on normal scales. The theory of rough paths considers streams of
information, such as $X$, for their effect on other systems and provides
quantitative tools to model this interaction. Consider the stream as the
input to an automata or controlled differential equation and so impacting on
the evolution on the state $Y$ of some controlled system:%
\begin{equation}
dY=f\left( Y\right) dX\text{, }Y_{0}=\xi .  \label{differential equation}
\end{equation}%
A key development of the theory is the development of quantitative tools and
estimates that allow one to estimate the response $Y$ from a top down
analysis of $X$ and in particular provides a mechanism for directly
quantifying the effects of the oscillatory components of $X$ without a
detailed analysis of the trajectory of $X$. As a result, the methods apply
to equations where $X$ does not have finite length. Differential equations
driven by Brownian motion can treated deterministically.

The interest in modelling and understanding such interactions is rather
wide. This manuscript is intended to create a useful interface by stating
and proving one of the main results in a way that appears to the authors
particularly useful for moving out into applications. It deliberately sets
out to hide the machinery and implementation of the main proofs in rough
path theory and to provide only a useful and rigorous statement of a result
that captures the essence of what the machinery delivers and is valid across
all Banach spaces (including finite dimensional ones) so that the methods
can be used more widely without great initial intellectual investment.

Davie \cite{Davie A. M} established some high order Euler estimates of
solution of rough differential equations, driven by $p$-rough paths, $1\leq
p<3$. By using geodesic approximations, Friz and Victoir \cite{Friz and
Victoir}, \cite{Friz and Victoir book} extend Davie's results to rough
differential equations driven by weak geometric $p$-rough paths, $p\geq 3$.
The formulation and proof in \cite{Davie A. M}, \cite{Friz and Victoir} and 
\cite{Friz and Victoir book} are dimension-dependent, and the error bound
may explode as the dimension increases.

By modifying the method used in \cite{Davie A. M}, \cite{Friz and Victoir}
and \cite{Friz and Victoir book}, we give a dimension-free high order Euler
estimation of solution of rough differential equations (i.e. both the
driving rough path and the solution path live in infinite dimensional
spaces). Our estimates are first developed for ordinary differential
equations. Then by passing to limit and using universal limit theorem (see 
\cite{Lyons1998}, \cite{LyonsQian}), similar estimates holds for rough
differential equations.

The main idea of our proof is to compare the solution of $\left( \ref%
{differential equation}\right) $ (on small interval $\left[ s,t\right] $)
with the solution of another ordinary differential equation (on $\left[ 0,1%
\right] $) whose vector field varies with $s,t$. Based on Arous \cite{Ben
Arous}, Hu \cite{Hu} and Castell \cite{Castell F}, the solution of
stochastic differential equation can be approximated on fixed small interval
by the exponential of a Lie series, and the exponential can equivalently be
treated as the solution of an ordinary differential equation. As we
demonstrate, their idea is also applicable to rough differential equations
in Banach spaces.

\section{Background and Notations}

We denote $\mathcal{U}$ and $\mathcal{V}$ as two Banach spaces.

\subsection{Algebraic Structure}

\begin{definition}
\label{Definition of tensor space and Lie bracket space}We select a norm on
tensor product of (elements in) $\mathcal{U}$ and $\mathcal{V}$, which
satisfies the inequality: (up to an universal constant) 
\begin{equation}
\left\Vert u\otimes v\right\Vert _{\mathcal{U}\otimes \mathcal{V}}\leq
\left\Vert u\right\Vert _{\mathcal{U}}\left\Vert v\right\Vert _{\mathcal{V}}%
\text{, \ }\forall u\in \mathcal{U}\text{, }\forall v\in \mathcal{V}\text{.}
\label{condition on tensor norm}
\end{equation}%
Define $\mathcal{U}\otimes \mathcal{V}$ and $\left[ \mathcal{U},\mathcal{V}%
\right] $ as the closure of 
\begin{gather*}
\left\{ \sum_{k=1}^{m}u_{k}\otimes v_{k}\text{, }\left\{ u_{k}\right\}
_{k=1}^{m}\subset \mathcal{U}\text{, }\left\{ v_{k}\right\}
_{k=1}^{m}\subset \mathcal{V}\text{, }m\geq 1\right\} \text{,} \\
\left\{ \sum_{k=1}^{m}\left( u_{k}\otimes v_{k}-v_{k}\otimes u_{k}\right) 
\text{, }\left\{ u_{k}\right\} _{k=1}^{m}\subset \mathcal{U}\text{, }\left\{
v_{k}\right\} _{k=1}^{m}\subset \mathcal{V}\text{, }m\geq 1\right\} \text{,}
\end{gather*}%
w.r.t. the norm selected on the tensor product $\otimes $.
\end{definition}

As an example, inequality $\left( \ref{condition on tensor norm}\right) $ is
satisfied by injective and projective tensor norms (Prop 2.1 and Prop 3.1 in 
\cite{Ryan}).

\begin{definition}
For integers $n\geq k\geq 1$, denote $\pi _{k}$ as the projection of $%
1\oplus \mathcal{V}\oplus \cdots \oplus \mathcal{V}^{\otimes n}$ to $%
\mathcal{V}^{\otimes k}$. Define $\exp _{n}:\mathcal{V}\oplus \cdots \oplus 
\mathcal{V}^{\otimes n}\rightarrow 1\oplus \mathcal{V}\oplus \cdots \oplus 
\mathcal{V}^{\otimes n}$ as%
\begin{equation*}
\exp _{n}\left( a\right) :=1+\sum_{k=1}^{n}\pi _{k}\left( \sum_{j=1}^{n}%
\frac{a^{\otimes j}}{j!}\right) \text{, }\forall a\in \mathcal{V}\oplus
\cdots \oplus \mathcal{V}^{\otimes n}\text{.}
\end{equation*}%
Define $\log _{n}:1\oplus \mathcal{V}\oplus \cdots \oplus \mathcal{V}%
^{\otimes n}\rightarrow \mathcal{V}\oplus \cdots \oplus \mathcal{V}^{\otimes
n}$ as%
\begin{equation*}
\log _{n}\left( g\right) :=\sum_{k=1}^{n}\pi _{k}\left( \sum_{j=1}^{n}\frac{%
\left( -1\right) ^{j+1}}{j}\left( g-1\right) ^{\otimes j}\right) \text{, }%
\forall g\in 1\oplus \mathcal{V}\oplus \cdots \oplus \mathcal{V}^{\otimes n}%
\text{.}
\end{equation*}
\end{definition}

\begin{definition}[$G^{n}\left( \mathcal{V}\right) $]
\label{Definition of nilpotent Lie group}Suppose $\mathcal{V}$ is a Banach
space. Then we define recursively%
\begin{equation}
\left[ \mathcal{V}\right] ^{k+1}:=\left[ \mathcal{V},\left[ \mathcal{V}%
\right] ^{k}\right] \text{ \ with \ }\left[ \mathcal{V}\right] ^{1}:=%
\mathcal{V},  \label{Notation of kth order Lie bracket of V}
\end{equation}%
and define%
\begin{equation*}
G^{n}\left( \mathcal{V}\right) :=\left\{ \exp _{n}\left( a\right) |a\in %
\left[ \mathcal{V}\right] ^{1}\oplus \cdots \oplus \left[ \mathcal{V}\right]
^{n}\right\} \text{.}
\end{equation*}%
For $g,h\in G^{n}\left( \mathcal{V}\right) $, we define product and inverse
as%
\begin{equation*}
g\otimes h:=\sum_{k=0}^{n}\left( \sum_{j=0}^{k}\pi _{j}\left( g\right)
\otimes \pi _{k-j}\left( h\right) \right) \text{ \ and \ }%
g^{-1}:=1+\sum_{k=1}^{n}\pi _{k}\left( \sum_{j=1}^{n}\left( -1\right)
^{j}\left( g-1\right) ^{\otimes j}\right) \text{.}
\end{equation*}%
If we equip $G^{n}\left( \mathcal{V}\right) $ with the homogeneous norm%
\begin{equation*}
\left\Vert g\right\Vert :=\sum_{k=1}^{n}\left\Vert \pi _{k}\left( g\right)
\right\Vert ^{\frac{1}{k}}\text{, }\forall g\in G^{n}\left( \mathcal{V}%
\right) \text{,}
\end{equation*}%
then $G^{n}\left( \mathcal{V}\right) $ is a topological group, called the
step-$n$ nilpotent Lie group over $\mathcal{V}$.
\end{definition}

\subsection{Rough Path}

\begin{definition}[$S_{n}\left( x\right) $]
\label{Definition of Sn(X)}Suppose $x:\left[ 0,T\right] \rightarrow \mathcal{%
V}$ is a continuous bounded variation path. For integer $n\geq 1$, define
the step-$n$ signature of $x$, $S_{n}\left( x\right) :\left[ 0,T\right]
\rightarrow \left( G^{n}\left( \mathcal{V}\right) ,\left\Vert \cdot
\right\Vert \right) $ by assigning that, for $t\in \left[ 0,T\right] $,%
\begin{equation*}
S_{n}\left( x\right) _{t}=\left(
1,x_{t}-x_{0},\iint_{0<u_{1}<u_{2}<t}dx_{u_{1}}\otimes dx_{u_{2}},\dots
,\idotsint\nolimits_{0<u_{1}<\dots <u_{n}<t}dx_{u_{1}}\otimes \cdots \otimes
dx_{u_{n}}\right) \text{.}
\end{equation*}
\end{definition}

\begin{definition}[$d_{p}$ metric and $p$-variation]
For $p\geq 1$, denote $\left[ p\right] $ as the integer part of $p$. Suppose 
$X$ and $Y$ are continuous paths defined on $\left[ 0,T\right] $ taking
value in $G^{\left[ p\right] }\left( \mathcal{V}\right) $. Define%
\begin{equation*}
d_{p}\left( X,Y\right) :=\max_{1\leq k\leq \left[ p\right] }\sup_{D\subset %
\left[ 0,T\right] }\left( \sum_{j,t_{j}\in D}\left\Vert \pi _{k}\left(
X_{t_{j},t_{j+1}}\right) -\pi _{k}\left( Y_{t_{j},t_{j+1}}\right)
\right\Vert ^{\frac{p}{k}}\right) ^{\frac{1}{p}}\text{,}
\end{equation*}%
where the supremum is taken over all finite partitions $D=\left\{
t_{j}\right\} _{j=0}^{n}$ of $\left[ 0,T\right] $ with $0=t_{0}<t_{1}<\cdots
<t_{n}=T$, $n\geq 1$. With $e$ denotes the identity path (i.e. $e_{t}=1\in
G^{\left[ p\right] }\left( \mathcal{V}\right) $, $t\in \left[ 0,T\right] $),
we define the $p$-variation of $X$ on $\left[ 0,T\right] $ as 
\begin{equation*}
\left\Vert X\right\Vert _{p-var,\left[ 0,T\right] }:=d_{p}\left( X,e\right) .
\end{equation*}
\end{definition}

\begin{definition}[geometric $p$-rough path]
\label{Definition of geometric rough path}$X:\left[ 0,T\right] \rightarrow
\left( G^{\left[ p\right] }\left( \mathcal{V}\right) ,\left\Vert \cdot
\right\Vert \right) $ is called a geometric $p$-rough path, if there exists
a sequence of continuous bounded variation paths $x_{l}:\left[ 0,T\right]
\rightarrow \mathcal{V}$, $l\geq 1$, such that 
\begin{equation*}
\lim_{l\rightarrow \infty }d_{p}\left( S_{\left[ p\right] }\left(
x_{l}\right) ,X\right) =0\text{.}
\end{equation*}
\end{definition}

\begin{definition}[$C^{\protect\gamma }\left( \mathcal{V},\mathcal{U}\right) 
$]
\label{Definition Cgamma}For $\gamma >0$, we say $r:\mathcal{V}\rightarrow 
\mathcal{U}$ is $Lip\left( \gamma \right) $ and denote $r\in C^{\gamma
}\left( \mathcal{V},\mathcal{U}\right) $, if and only if $r$ is $\lfloor
\gamma \rfloor $-times Fr\'{e}chet differentiable ($\lfloor \gamma \rfloor $
denotes the largest integer which is strictly less than $\gamma $), and%
\begin{equation*}
\left\vert r\right\vert _{Lip\left( \gamma \right) }:=\max_{0\leq k\leq
\lfloor \gamma \rfloor }\left\Vert D^{k}r\right\Vert _{\infty }\vee
\left\Vert D^{\lfloor \gamma \rfloor }r\right\Vert _{\left( \gamma -\lfloor
\gamma \rfloor \right) -H\ddot{o}l}<\infty \text{,}
\end{equation*}%
where $\left\Vert \cdot \right\Vert _{\infty }$ denotes the uniform norm and 
$\left\Vert \cdot \right\Vert _{\left( \gamma -\lfloor \gamma \rfloor
\right) -H\ddot{o}l}$ denotes the $\left( \gamma -\lfloor \gamma \rfloor
\right) $-H\"{o}lder norm.

\noindent Denote $C^{0}\left( \mathcal{V},\mathcal{U}\right) $ as the space
of bounded measurable mappings from $\mathcal{V}$ to $\mathcal{U}$.
\end{definition}

\begin{definition}[$L\left( \mathcal{W},C^{\protect\gamma }\left( \mathcal{V}%
,\mathcal{U}\right) \right) $]
\label{Definition LV Cgamma}Suppose $\mathcal{U}$, $\mathcal{V}$ and $%
\mathcal{W}$ are Banach spaces. Denote $L\left( \mathcal{W},C^{\gamma
}\left( \mathcal{V},\mathcal{U}\right) \right) $ as the space of linear
mappings from $\mathcal{W}$ to $C^{\gamma }\left( \mathcal{V},\mathcal{U}%
\right) $, and denote%
\begin{equation*}
\left\vert f\right\vert _{Lip\left( \gamma \right) }:=\sup_{w\in \mathcal{W}%
,\left\Vert w\right\Vert =1}\left\vert f\left( w\right) \right\vert
_{Lip\left( \gamma \right) }\text{, }\forall f\in L\left( \mathcal{W}%
,C^{\gamma }\left( \mathcal{V},\mathcal{U}\right) \right) \text{.}
\end{equation*}
\end{definition}

We define solution of rough differential equation as in Def 5.1 in Lyons 
\cite{Lyonsnotes}.

\begin{definition}[solution of RDE]
\label{Definition of solution of RDE}Suppose $\mathcal{U}$ and $\mathcal{V}$
are two Banach spaces, $X:\left[ 0,T\right] \rightarrow \left( G^{\left[ p%
\right] }\left( \mathcal{V}\right) ,\left\Vert \cdot \right\Vert \right) $
is a geometric $p$-rough path, $f\in L\left( \mathcal{V},C^{\gamma }\left( 
\mathcal{U},\mathcal{U}\right) \right) $ for some $\gamma >p-1$ and $\xi \in 
\mathcal{U}$. Define $h:\mathcal{V}\oplus \mathcal{U}\rightarrow L\left( 
\mathcal{V}\oplus \mathcal{U},\mathcal{V}\oplus \mathcal{U}\right) $ as%
\begin{equation}
h\left( v_{1},u_{1}\right) \left( v_{2},u_{2}\right) =\left( v_{2},f\left(
v_{2}\right) \left( u_{1}+\xi \right) \right) \text{, \ }\forall
v_{1},v_{2}\in \mathcal{V}\text{, }\forall u_{1},u_{2}\in \mathcal{U}\text{.}
\label{definition of h}
\end{equation}%
Then geometric $p$-rough path $Z:\left[ 0,T\right] \rightarrow \left( G^{%
\left[ p\right] }\left( \mathcal{V}\oplus \mathcal{U}\right) ,\left\Vert
\cdot \right\Vert \right) \,$\ is said to be a solution to the rough
differential equation%
\begin{equation}
dY=f\left( Y\right) dX\text{, }Y_{0}=\xi \text{,}  \label{RDE}
\end{equation}%
if $\pi _{G^{\left[ p\right] }\left( \mathcal{V}\right) }\left( Z\right) =X$%
, and $Z$ satisfies the rough integral equation (in sense of Def 4.9 \cite%
{Lyonsnotes}):%
\begin{equation*}
Z_{t}=\int_{0}^{t}h\left( Z_{u}\right) dZ_{u}\text{, \ }t\in \left[ 0,T%
\right] \text{.}
\end{equation*}
\end{definition}

For $g\in G^{n}\left( \mathcal{V}\right) $ and $\lambda >0$, we denote $%
\delta _{\lambda }g:=1+\sum_{k=1}^{n}\lambda ^{k}\pi _{k}\left( g\right) $.

\begin{theorem}[Lyons]
\label{Theorem Lyons existence and uniqueness}When $f$ in $\left( \ref{RDE}%
\right) $ is in $L\left( \mathcal{V},C^{\gamma }\left( \mathcal{U},\mathcal{U%
}\right) \right) $ for $\gamma >p$, the solution of $\left( \ref{RDE}\right) 
$ exists uniquely. Moreover, there exists constant $C_{p,\gamma }$, such
that, for any interval $\left[ s,t\right] \subseteq \left[ 0,T\right] $
satisfying $\left\vert f\right\vert _{Lip\left( \gamma \right) }\left\Vert
X\right\Vert _{p-var,\left[ s,t\right] }\leq 1$, we have (after rescaling $f$%
, $X$ and $Y$)%
\begin{equation}
\left\Vert \left( \delta _{\left\vert f\right\vert _{Lip\left( \gamma
\right) }}X,Y\right) \right\Vert _{p-var,\left[ s,t\right] }\leq C_{p,\gamma
}\left\vert f\right\vert _{Lip\left( \gamma \right) }\left\Vert X\right\Vert
_{p-var,\left[ s,t\right] }\text{.}  \label{estimation of solution of RDE}
\end{equation}
\end{theorem}

\begin{remark}
The unique solution of $\left( \ref{RDE}\right) $ is recovered by a sequence
of rough integrals. Then based on Thm 4.12 and Prop 5.9 in \cite{Lyonsnotes}
and by using lower semi-continuity of $p$-variation, the constant $%
C_{p,\gamma }$ in $\left( \ref{estimation of solution of RDE}\right) $ is an
absolute constant which only depends on $p$ and $\gamma $, and is finite
whenever $\gamma >p-1$.
\end{remark}

When $\mathcal{U}$ and $\mathcal{V}$ are finite dimensional spaces, for any $%
f\in L\left( \mathcal{V},C^{\gamma }\left( \mathcal{U},\mathcal{U}\right)
\right) $, $\gamma >p-1$, there exists a solution to $\left( \ref{RDE}%
\right) $ which satisfies $\left( \ref{estimation of solution of RDE}\right) 
$. Indeed, based on Prop 5.9 \cite{Lyonsnotes}, when $f$ is $Lip\left(
\gamma \right) $ for $\gamma >p-1$, the sequence of Picard iterations $%
\left\{ Z^{n}\right\} _{n}:\left[ 0,T\right] \rightarrow G^{\left[ p\right]
}\left( \mathcal{V}\oplus \mathcal{U}\right) $, define recursively as rough
integrals: (with $h$ defined at $\left( \ref{definition of h}\right) $)%
\begin{eqnarray*}
Z_{t}^{0} &=&\left( X_{t},0\right) \text{, \ }t\in \left[ 0,T\right] \text{,}
\\
Z_{t}^{n+1} &=&\int_{0}^{t}h\left( Z_{u}^{n}\right) dZ_{u}^{n}\text{, \ }%
t\in \left[ 0,T\right] \text{, \ }n\geq 0\text{,}
\end{eqnarray*}%
are uniformly bounded and equi-continuous. When $\mathcal{V}$ and $\mathcal{U%
}$ are finite-dimensional spaces, bounded sets in $G^{\left[ p\right]
}\left( \mathcal{V}\oplus \mathcal{U}\right) $ are relatively compact. Thus,
based on Arzel\`{a}-Ascoli theorem, there exists a subsequence $\left\{
Z^{n_{k}}\right\} _{k}$ which converge uniformly (denoted the limit as $Z$).
Then by spelling out the almost-multiplicative functional (associated with
the Picard iteration) and letting $k$ tends to infinity, one can prove that $%
Z$ is a solution to the rough differential equation $\left( \ref{RDE}\right) 
$. Then based on Thm 4.12 and Prop 5.9 in \cite{Lyonsnotes} and by using
lower semi-continuity of $p$-variation, the estimate $\left( \ref{estimation
of solution of RDE}\right) $ holds for $Z$.

When $\mathcal{U}$ is a Banach space and when $f$ in $\left( \ref{RDE}%
\right) $ is $Lip\left( \gamma \right) $ for $\gamma \in \left( p-1,p\right) 
$, there does not always exist a solution to $\left( \ref{RDE}\right) $.
Godunov \cite{Godunov} proved that, "each Banach space in which Peano's
theorem is true is finite-dimensional". Shkarin \cite{Shkarin} (in Cor 1.5)
proved that, for any real infinite dimensional Banach space (denoted as $%
\mathcal{V}$), which has a complemented subspace with an unconditional
Schauder basis, and for any $\alpha \in \left( 0,1\right) $, there exists $%
\alpha $-H\"{o}lder continuous function $f:\mathcal{V}\rightarrow \mathcal{V}
$, such that the equation $\dot{x}=f\left( x\right) $ has no solution in any
interval of the real line. Based on Shkarin \cite{Shkarin} (Rrk 1.4), $L_{p}%
\left[ 0,1\right] $ ($1\leq p<\infty $) and $C\left[ 0,1\right] $ are
examples of such Banach spaces, and "roughly speaking, all infinite
dimensional Banach spaces, which naturally appear in analysis" fall into
this category.

\subsection{Differential Operator}

For $\gamma \geq 0$, recall $C^{\gamma }\left( \mathcal{U},\mathcal{U}%
\right) $ in Definition \ref{Definition Cgamma}, and that $L\left( \mathcal{V%
},C^{\gamma }\left( \mathcal{U},\mathcal{U}\right) \right) $ denotes the
space of linear mappings from $\mathcal{V}$ to $C^{\gamma }\left( \mathcal{U}%
,\mathcal{U}\right) $. With $f\in L\left( \mathcal{V},C^{\gamma }\left( 
\mathcal{U},\mathcal{U}\right) \right) $ and integer $k\leq \gamma +1$, we
clarify the meaning of differential operator $f^{\circ k}\left( v\right) $
for $v\in \mathcal{V}^{\otimes k}$. For $r\in C^{k}\left( \mathcal{U},%
\mathcal{U}\right) $ and $j=0,1,\dots ,k$, $D^{j}r\in L\left( \mathcal{U}%
^{\otimes j},C^{k-j}\left( \mathcal{U},\mathcal{U}\right) \right) $.

\begin{notation}[$\mathcal{D}^{k}\left( \mathcal{U}\right) $]
For integer $k\geq 0$, denote $\mathcal{D}^{k}\left( \mathcal{U}\right) $ as
the set of $k$th order differential operators (on $Lip\left( k\right) $
functions from $\mathcal{U}$ to $\mathcal{U}$). More specifically, $p\in 
\mathcal{D}^{k}\left( \mathcal{U}\right) $ if and only if $p:C^{k}\left( 
\mathcal{U},\mathcal{U}\right) \rightarrow C^{0}\left( \mathcal{U},\mathcal{U%
}\right) $ and there exist bounded $p^{j}:\mathcal{U}\rightarrow \mathcal{U}%
^{\otimes j}$, $j=0,1,\dots ,k$, with $p_{k}\not\equiv 0$, such that%
\begin{equation*}
p\left( r\right) \left( u\right) =\sum_{j=0}^{k}\left( D^{j}r\right) \left(
p_{j}\left( u\right) \right) \left( u\right) \text{, }\forall u\in \mathcal{U%
}\text{, }\forall r\in C^{k}\left( \mathcal{U},\mathcal{U}\right) \text{.}
\end{equation*}%
We define norm $\left\vert \cdot \right\vert _{k}$ on $\mathcal{D}^{k}\left( 
\mathcal{U}\right) $ as%
\begin{equation*}
\left\vert p\right\vert _{k}:=\max_{j=0,1,\dots ,k}\sup_{u\in \mathcal{U}%
}\left\Vert p_{j}\left( u\right) \right\Vert \text{, }\forall p\in \mathcal{D%
}^{k}\left( \mathcal{U}\right) \text{.}
\end{equation*}%
Then $\mathcal{D}^{k}\left( \mathcal{U}\right) $ can be extended to a Banach
space $\left( \mathcal{D}^{k}\left( \mathcal{U}\right) ,\left\vert \cdot
\right\vert _{k}\right) $ (with the natural addition and scalar
multiplication).
\end{notation}

\begin{definition}[composition]
\label{Definition differential operator f}Suppose $p^{1}\in \mathcal{D}%
^{j_{1}}\left( \mathcal{U}\right) $ and $p^{2}\in \mathcal{D}^{j_{2}}\left( 
\mathcal{U}\right) $ for integers $j_{1}\geq 0$, $j_{2}\geq 0$. Define the
composition of $p^{1}\circ p^{2}\in \mathcal{D}^{j_{1}+j_{2}}\left( \mathcal{%
U}\right) $ as%
\begin{equation*}
\left( p^{1}\circ p^{2}\right) \left( r\right) :=p^{1}\left( p^{2}\left(
r\right) \right) \text{, }\forall r\in C^{j_{1}+j_{2}}\left( \mathcal{U},%
\mathcal{U}\right) \text{.}
\end{equation*}%
When $p\in \mathcal{D}^{j}\left( \mathcal{U}\right) $, $j\geq 0$, we define
the differential operator $p^{\circ k}\in \mathcal{D}^{k\times j}\left( 
\mathcal{U}\right) $ for integer $k\geq 1$ as 
\begin{equation}
p^{\circ 1}:=p\text{ \ and \ }p^{\circ k}:=p\circ p^{\circ \left( k-1\right)
}\text{, }k\geq 2\text{.}  \label{Notation of composition of f with itself}
\end{equation}
\end{definition}

Composition of differential operators is associative, i.e. $\left(
p^{1}\circ p^{2}\right) \circ p^{3}=p^{1}\circ \left( p^{2}\circ
p^{3}\right) $.

\begin{definition}[$f^{\circ k}$]
\label{Definition of f circ k}Suppose $f\in L\left( \mathcal{V},C^{\gamma
}\left( \mathcal{U},\mathcal{U}\right) \right) $ for some $\gamma \geq 0$.
Then for any $v\in \mathcal{V}$, we treat $f\left( v\right) $ as a first
order differential operator (i.e. in $\mathcal{D}^{1}\left( \mathcal{U}%
\right) $), and define%
\begin{equation*}
f\left( v\right) \left( r\right) \left( u\right) :=\left( Dr\right) \left(
f\left( v\right) \left( u\right) \right) \left( u\right) \text{, }\forall
u\in \mathcal{U}\text{, }\forall r\in C^{1}\left( \mathcal{U},\mathcal{U}%
\right) \text{.}
\end{equation*}%
For integer $k\in 1,2,\dots ,\left[ \gamma \right] +1$ and $\left\{
v_{j}\right\} _{j=1}^{k}\subset \mathcal{V}$, we define $f^{\circ k}\left(
v_{1}\otimes \cdots \otimes v_{k}\right) \in \mathcal{D}^{k}\left( \mathcal{U%
}\right) $ as 
\begin{equation}
f^{\circ k}\left( v_{1}\otimes \cdots \otimes v_{k}\right) :=f\left(
v_{1}\right) \circ f\left( v_{2}\right) \circ \cdots \circ f\left(
v_{k}\right) .  \label{definition of fcirc k on elements}
\end{equation}%
Then we denote $f^{\circ k}\in L\left( \mathcal{V}^{\otimes k},\left( 
\mathcal{D}^{k}\left( \mathcal{U}\right) ,\left\vert \cdot \right\vert
_{k}\right) \right) $ as the unique continuous linear operator satisfying $%
\left( \ref{definition of fcirc k on elements}\right) $.
\end{definition}

\section{Main Result}

In this manuscript, we work with the first level (or "path" level) solution
of rough differential equations.

Firstly, we prove a lemma for ordinary differential equations. Then after
applying universal limit theorem (Thm 5.3 \cite{Lyonsnotes}), this lemma
leads to similar estimates of rough differential equations. The proof of
this lemma is in the same spirit as Lemma 2.4(a) in \cite{Davie A. M}, Lemma
16 in \cite{Friz and Victoir} and Lemma 10.7 in \cite{Friz and Victoir book}%
, only that we use the ordinary differential equation $\left( \ref%
{Definition of approximating ode}\right) $ for the approximation.

We denote $\mathcal{U}$ and $\mathcal{V}$ as two Banach spaces. For $p\geq 1$%
, denote $\left[ p\right] $ as the integer part of $p$. For $\gamma >0$,
denote $\lfloor \gamma \rfloor $ as the largest integer which is strictly
less than $\gamma $. Denote $I_{d}:\mathcal{U}\rightarrow \mathcal{U}\ $as
the identity function, i.e. $I_{d}\left( u\right) =u$, $\forall u\in 
\mathcal{U}$. For $f\in L\left( \mathcal{V},C^{\gamma }\left( \mathcal{U},%
\mathcal{U}\right) \right) $, integer $k\leq \gamma +1$ and $v\in \mathcal{V}%
^{\otimes k}$, recall the differential operator $f^{\circ k}\left( v\right) $
defined in Definition \ref{Definition of f circ k}.

\begin{lemma}
\label{Lemma important}Suppose $\mathcal{U}$ and $\mathcal{V}$ are two
Banach spaces, $x:\left[ 0,T\right] \rightarrow \mathcal{V}$ is a continuous
bounded variation path, $f\in L\left( \mathcal{V},C^{\gamma }\left( \mathcal{%
U},\mathcal{U}\right) \right) $ for $\gamma >1$, and $\xi \in \mathcal{U}$.
Denote $y:\left[ 0,T\right] \rightarrow \mathcal{U}$ as the\ unique solution
to the ordinary differential equation%
\begin{equation}
dy=f\left( y\right) dx\text{, \ }y_{0}=\xi \in \mathcal{U}\text{.}
\label{ODE}
\end{equation}%
Then for any $p\in \lbrack 1,\gamma +1)$, there exists a constant $%
C_{p,\gamma }$, which only depends on $p$ and $\gamma $, such that, for any $%
0\leq s<t\leq T$, if we denote $y^{s,t}:\left[ 0,1\right] \rightarrow 
\mathcal{U}$ as the unique solution of the ordinary differential equation:
(with $y_{s}$ denotes the value of $y$ in $\left( \ref{ODE}\right) $ at
point $s$) 
\begin{eqnarray}
dy_{u}^{s,t} &=&\left( \sum_{k=1}^{\lfloor \gamma \rfloor }f^{\circ k}\pi
_{k}\left( \log _{\lfloor \gamma \rfloor +1}\left( S_{\lfloor \gamma \rfloor
+1}\left( x\right) _{s,t}\right) \right) \left( I_{d}\right) \left(
y_{u}^{s,t}\right) \right) du\text{, }u\in \left[ 0,1\right] \text{, }
\label{Definition of approximating ode} \\
y_{0}^{s,t} &=&y_{s}+f^{\circ \left( \lfloor \gamma \rfloor +1\right) }\pi
_{\lfloor \gamma \rfloor +1}\left( \log _{\lfloor \gamma \rfloor +1}\left(
S_{\lfloor \gamma \rfloor +1}\left( x\right) _{s,t}\right) \right) \left(
I_{d}\right) \left( y_{s}\right) \text{,}  \notag
\end{eqnarray}%
then 
\begin{gather}
\left( 1\right) ,\left\Vert y_{t}-y_{1}^{s,t}\right\Vert \leq C_{p,\gamma
}\left\vert f\right\vert _{Lip\left( \gamma \right) }^{\gamma +1}\left\Vert
S_{\left[ p\right] }\left( x\right) \right\Vert _{p-var,\left[ s,t\right]
}^{\gamma +1}\text{,}  \label{ODE approximation of main RDE} \\
\left( 2\right) ,\left\Vert y_{t}-y_{s}-\sum_{k=1}^{\lfloor \gamma \rfloor
+1}f^{\circ k}\pi _{k}\left( S_{\lfloor \gamma \rfloor +1}\left( x\right)
_{s,t}\right) \left( I_{d}\right) \left( y_{s}\right) \right\Vert \leq
C_{p,\gamma }\left\vert f\right\vert _{Lip\left( \gamma \right) }^{\gamma
+1}\left\Vert S_{\left[ p\right] }\left( x\right) \right\Vert _{p-var,\left[
s,t\right] }^{\gamma +1}\text{.}  \notag
\end{gather}
\end{lemma}

The proof of Lemma \ref{Lemma important} starts from page \pageref{Proof of
Lemma important}. Since $f\in L\left( \mathcal{V},C^{\gamma }\left( \mathcal{%
U},\mathcal{U}\right) \right) $, it might be more appropriate to write $%
\left( \ref{ODE}\right) $ as $dy=f\left( dx\right) \left( y\right) $. We
keep it in the current form so that it is consistent with the classical
notation of ordinary differential equations.

\begin{remark}
We used $\left( \ref{Definition of approximating ode}\right) $ instead of
the ordinary differential equation 
\begin{eqnarray}
d\widetilde{y}_{u}^{s,t} &=&\left( \sum_{k=1}^{\lfloor \gamma \rfloor
+1}f^{\circ k}\pi _{k}\left( \log _{\lfloor \gamma \rfloor +1}\left(
S_{\lfloor \gamma \rfloor +1}\left( x\right) _{s,t}\right) \right) \left(
I_{d}\right) \left( \widetilde{y}_{u}^{s,t}\right) \right) du\text{, }u\in %
\left[ 0,1\right] \text{,}  \label{inner another ode} \\
\widetilde{y}_{0}^{s,t} &=&y_{s}\text{,}  \notag
\end{eqnarray}%
because $\mathcal{U}$ is a Banach space, and $\left( \ref{inner another ode}%
\right) $ may not has a solution (Cor 1.5 in Shkarin \cite{Shkarin}). If $%
\left( \ref{inner another ode}\right) $ has a solution (e.g. when $\mathcal{U%
}$ is finite dimensional), then $\left( \ref{ODE approximation of main RDE}%
\right) $ holds with $y_{1}^{s,t}$ replaced by $\widetilde{y}_{1}^{s,t}$.
\end{remark}

\begin{remark}
When $\mathcal{V}$ is $%
\mathbb{R}
^{d}$ and $\mathcal{U}$ is $%
\mathbb{R}
^{e}$ in Lemma \ref{Lemma important}, suppose $f=\left( f^{1},\dots
,f^{d}\right) \in L\left( 
\mathbb{R}
^{d},C^{\gamma }\left( 
\mathbb{R}
^{e},\mathcal{%
\mathbb{R}
}^{e}\right) \right) $. For $i=1,\dots ,d$, we treat $f^{i}=\left(
f_{1}^{i},\dots ,f_{e}^{i}\right) $ as a first order differential operator: $%
\sum_{j=1}^{e}f_{j}^{i}\frac{\partial }{\partial y_{j}}$. Then it can be
checked that (with $x=\left( x^{1},\dots ,x^{d}\right) :\left[ 0,T\right]
\rightarrow 
\mathbb{R}
^{d}$)%
\begin{equation*}
f^{\circ k}\pi _{k}\left( S_{\lfloor \gamma \rfloor +1}\left( x\right)
_{s,t}\right) \left( I_{d}\right) =\sum_{i_{1},\dots ,i_{k}\in \left\{
1,\dots ,d\right\} }\left( f^{i_{1}}\circ \cdots \circ f^{i_{k}}\right)
\left( I_{d}\right) \idotsint\nolimits_{s<u_{1}<\cdots
<u_{k}<t}dx_{u_{1}}^{i_{1}}\cdots dx_{u_{k}}^{i_{k}}\text{,}
\end{equation*}%
and our formulation coincides with the formulation in \cite{Davie A. M}, 
\cite{Friz and Victoir} and \cite{Friz and Victoir book}.
\end{remark}

The theorem below follows from universal limit theorem (Thm 5.3 \cite%
{Lyonsnotes}) and Lemma \ref{Lemma important}. Suppose $X:\left[ 0,T\right]
\rightarrow G^{\left[ p\right] }\left( \mathcal{V}\right) $ is a geometric $%
p $-rough path. For integer $n\geq \left[ p\right] $, we denote $S_{n}\left(
X\right) \ $as the unique enhancement of $X$ to a continuous path with
finite $p$-variation taking value in $G^{n}\left( \mathcal{V}\right) $ (Thm
3.7 \cite{Lyonsnotes}).

\begin{theorem}
\label{Theorem rough Euler expansion of solution of RDE}Suppose $f\in
L\left( \mathcal{V},C^{\gamma }\left( \mathcal{U},\mathcal{U}\right) \right) 
$ for $\gamma >1$, $X:\left[ 0,T\right] \rightarrow G^{\left[ p\right]
}\left( \mathcal{V}\right) $ is a geometric $p$-rough path for some $p\in
\lbrack 1,\gamma )$, and $\xi \in \mathcal{U}$. Denote $Z$ as the unique
solution (in the sense of Definition \ref{Definition of solution of RDE}) of
the rough differential equation%
\begin{equation}
dY=f\left( Y\right) dX\text{, \ }Y_{0}=\xi \text{.}  \label{main RDE}
\end{equation}
Denote $Y:=\pi _{G^{\left[ p\right] }\left( \mathcal{U}\right) }\left(
Z\right) $. Then there exists a constant $C_{p,\gamma }$, which only depends
on $p$ and $\gamma $, such that, for any $0\leq s\leq t\leq T$, if denote $%
y^{s,t}:\left[ 0,1\right] \rightarrow \mathcal{U}$ as the unique solution of
the ordinary differential equation: 
\begin{eqnarray*}
dy_{u}^{s,t} &=&\left( \sum_{k=1}^{\lfloor \gamma \rfloor }f^{\circ k}\pi
_{k}\left( \log _{\lfloor \gamma \rfloor +1}\left( S_{\lfloor \gamma \rfloor
+1}\left( X\right) _{s,t}\right) \right) \left( I_{d}\right) \left(
y_{u}^{s,t}\right) \right) du\text{, }u\in \left[ 0,1\right] \text{, } \\
y_{0}^{s,t} &=&\pi _{1}\left( Y_{s}\right) +f^{\circ \left( \lfloor \gamma
\rfloor +1\right) }\pi _{\lfloor \gamma \rfloor +1}\left( \log _{\lfloor
\gamma \rfloor +1}\left( S_{\lfloor \gamma \rfloor +1}\left( X\right)
_{s,t}\right) \right) \left( I_{d}\right) \left( \pi _{1}\left( Y_{s}\right)
\right) \text{,}
\end{eqnarray*}%
then we have 
\begin{gather*}
\left( 1\right) ,\left\Vert \pi _{1}\left( Y_{t}\right)
-y_{1}^{s,t}\right\Vert \leq C_{p,\gamma }\left\vert f\right\vert
_{Lip\left( \gamma \right) }^{\gamma +1}\left\Vert X\right\Vert _{p-var,
\left[ s,t\right] }^{\gamma +1}\text{,} \\
\left( 2\right) ,\left\Vert \pi _{1}\left( Y_{s,t}\right)
-\sum_{k=1}^{\lfloor \gamma \rfloor +1}f^{\circ k}\pi _{k}\left( S_{\lfloor
\gamma \rfloor +1}\left( X\right) _{s,t}\right) \left( I_{d}\right) \left(
\pi _{1}\left( Y_{s}\right) \right) \right\Vert \leq C_{p,\gamma }\left\vert
f\right\vert _{Lip\left( \gamma \right) }^{\gamma +1}\left\Vert X\right\Vert
_{p-var,\left[ s,t\right] }^{\gamma +1}\text{.}
\end{gather*}
\end{theorem}

The proof of Theorem \ref{Theorem rough Euler expansion of solution of RDE}\
is on page \pageref{Proof of Theorem rough Euler expansion of solution of
RDE}.

\begin{remark}
Suppose $f\in L\left( \mathcal{V},C^{\gamma }\left( \mathcal{U},\mathcal{U}%
\right) \right) $ for $\gamma >1$, $X:\left[ 0,T\right] \rightarrow G^{\left[
p\right] }\left( \mathcal{V}\right) $ is a weak geometric $p$-rough path%
\footnote{%
Continuous path $X:\left[ 0,T\right] \rightarrow G^{\left[ p\right] }\left( 
\mathcal{V}\right) $ is called a weak geometric $p$-rough path, if $%
\left\Vert X\right\Vert _{p-var,\left[ 0,T\right] }<\infty $.} for some $%
p\in \lbrack 1,\gamma +1)$, and $\xi \in \mathcal{U}$. Then by following
similar reasoning as that of Theorem \ref{Theorem rough Euler expansion of
solution of RDE}, one can prove that, any solution, in the sense of Def
10.17 in \cite{Friz and Victoir book}, to the rough differential equation 
\begin{equation*}
dy=f\left( y\right) dX\text{, }y_{0}=\xi \text{,}
\end{equation*}%
satisfies the estimates in Theorem \ref{Theorem rough Euler expansion of
solution of RDE}.
\end{remark}

The theorem below follows from Thm 4.12 \cite{Lyonsnotes} and Lemma \ref%
{Lemma important}.

\begin{theorem}
\label{Theorem Euler expansion of rough integral}Suppose $f\in L\left( 
\mathcal{V},C^{\gamma }\left( \mathcal{V},\mathcal{U}\right) \right) $ for $%
\gamma >1$, and $X$ is a geometric $p$-rough path for some $p\in \lbrack
1,\gamma +1)$. Denote $Y:\left[ 0,T\right] \rightarrow G^{\left[ p\right]
}\left( \mathcal{U}\right) $ as the rough integral (in the sense of Def 4.9 
\cite{Lyonsnotes}) 
\begin{equation*}
Y_{t}=\int_{0}^{t}f\left( X\right) dX\text{, }t\in \left[ 0,T\right] \text{.}
\end{equation*}%
Then there exists constant $C_{p,\gamma }$, which only depends on $p$ and $%
\gamma $ and is finite whenever $\gamma >p-1$, such that, for any $0\leq
s\leq t\leq T$, if denote $y^{s,t}:\left[ 0,T\right] \rightarrow \mathcal{U}$
as the unique solution of the ordinary differential equation: 
\begin{eqnarray*}
dy_{u}^{s,t} &=&\left( \sum_{k=1}^{\lfloor \gamma \rfloor }\left(
D^{k-1}f\right) \pi _{k}\left( \log _{\lfloor \gamma \rfloor +1}\left(
S_{\lfloor \gamma \rfloor +1}\left( X\right) _{s,t}\right) \right) \left(
y_{u}^{s,t}\right) \right) du\text{, }u\in \left[ 0,1\right] \text{, } \\
y_{0}^{s,t} &=&\pi _{1}\left( Y_{s}\right) +\left( D^{\lfloor \gamma \rfloor
}f\right) \pi _{\lfloor \gamma \rfloor +1}\left( \log _{\lfloor \gamma
\rfloor +1}\left( S_{\lfloor \gamma \rfloor +1}\left( X\right) _{s,t}\right)
\right) \left( \pi _{1}\left( Y_{s}\right) \right) \text{,}
\end{eqnarray*}%
then we have%
\begin{gather*}
\left( 1\right) ,\left\Vert \pi _{1}\left( Y_{t}\right)
-y_{1}^{s,t}\right\Vert \leq C_{p,\gamma }\left\vert f\right\vert
_{Lip\left( \gamma \right) }^{\gamma +1}\left\Vert X\right\Vert _{p-var,
\left[ s,t\right] }^{\gamma +1}\text{,} \\
\left( 2\right) ,\left\Vert \pi _{1}\left( Y_{s,t}\right)
-\sum_{k=1}^{\lfloor \gamma \rfloor +1}\left( D^{k-1}f\right) \pi _{k}\left(
S_{\lfloor \gamma \rfloor +1}\left( X\right) _{s,t}\right) \left( \pi
_{1}\left( Y_{s}\right) \right) \right\Vert \leq C_{p,\gamma }\left\vert
f\right\vert _{Lip\left( \gamma \right) }^{\gamma +1}\left\Vert X\right\Vert
_{p-var,\left[ s,t\right] }^{\gamma +1}\text{.}
\end{gather*}
\end{theorem}

The proof of Theorem \ref{Theorem Euler expansion of rough integral} is
almost the same as that of Theorem \ref{Theorem rough Euler expansion of
solution of RDE}, only that we estimate the rough differential equation 
\begin{equation*}
d\left( X,Y\right) =\left( 1_{\mathcal{V}},f\left( X\right) \right) dX\text{%
, \ }\left( X,Y\right) _{0}=\left( X_{0},0\right) \in \mathcal{V}\oplus 
\mathcal{U}\text{,}
\end{equation*}
and uses Thm 4.12 \cite{Lyonsnotes} instead of universal limit theorem.

\section{Proof}

We made explicit the dependence of constants (e.g. $C_{p,\gamma }$), but the
exact value of constants may change from line to line. We denote $\mathcal{U}
$ and $\mathcal{V}$ as two Banach spaces. Recall $C^{\gamma }\left( \mathcal{%
U},\mathcal{U}\right) $ in Definition \ref{Definition Cgamma} on page %
\pageref{Definition Cgamma}.

\begin{lemma}
\label{Lemma exchange of Id and dot}Suppose $k\geq 1$ is an integer, and $%
f\in L\left( \mathcal{V},C^{k-1}\left( \mathcal{U},\mathcal{U}\right)
\right) $. Recall $f^{\circ k}\in L\left( \mathcal{V}^{\otimes k},\left( 
\mathcal{D}^{k}\left( \mathcal{U}\right) ,\left\vert \cdot \right\vert
_{k}\right) \right) $ defined in Definition \ref{Definition of f circ k} (on
page \pageref{Definition of f circ k}). Then for any $v\in \left[ \mathcal{V}%
\right] ^{k}$ (denoted at $\left( \ref{Notation of kth order Lie bracket of
V}\right) $ on page \pageref{Notation of kth order Lie bracket of V}), $%
f^{\circ k}\left( v\right) $ is a first order differential operator, which
satisfies (with $I_{d}:\mathcal{U}\rightarrow \mathcal{U}$ denotes the
identity function)%
\begin{equation*}
f^{\circ k}\left( v\right) \left( r\right) =\left( Dr\right) \left( f^{\circ
k}\left( v\right) \left( I_{d}\right) \right) \text{, \ }\forall r\in
C^{1}\left( \mathcal{U},\mathcal{U}\right) \text{.}
\end{equation*}
\end{lemma}

\begin{proof}
To prove that $f^{\circ k}\left( v\right) $ is a first order differential
operator, we define another first order differential operator, and want to
prove that they are the same.

Suppose $\mathcal{V}^{1}$ and $\mathcal{V}^{2}$ are two Banach spaces, and $%
f^{i}\in L\left( \mathcal{V}^{i},\left( \mathcal{D}^{k_{i}}\left( \mathcal{U}%
\right) ,\left\vert \cdot \right\vert _{k_{i}}\right) \right) $, $i=1,2$.
Define $\left[ f^{1},f^{2}\right] \in L\left( \mathcal{V}^{1}\otimes 
\mathcal{V}^{2},\left( \mathcal{D}^{1}\left( \mathcal{U}\right) ,\left\vert
\cdot \right\vert _{1}\right) \right) $ as the unique continuous linear
operator, which satisfies that, for any $v^{1}\in \mathcal{V}^{1}$, any $%
v^{2}\in \mathcal{V}^{2}$, and any $r\in C^{1}\left( \mathcal{U},\mathcal{U}%
\right) $,%
\begin{equation}
\left[ f^{1},f^{2}\right] \left( v^{1}\otimes v^{2}\right) \left( r\right)
=\left( Dr\right) \left( f^{1}\left( v^{1}\right) \circ f^{2}\left(
v^{2}\right) -f^{2}\left( v^{2}\right) \circ f^{1}\left( v^{1}\right)
\right) \left( I_{d}\right) \text{.}
\label{inner definition of bracket of differential operator}
\end{equation}%
For integer $k\geq 1$ and $f\in L\left( \mathcal{V},C^{k-1}\left( \mathcal{U}%
,\mathcal{U}\right) \right) $, define $\left[ f\right] ^{\circ k}\in L\left( 
\mathcal{V}^{\otimes k},\left( \mathcal{D}^{1}\left( \mathcal{U}\right)
,\left\vert \cdot \right\vert _{1}\right) \right) $ as (with $f^{\circ 1}$
defined in Definition \ref{Definition of f circ k})%
\begin{equation}
\left[ f\right] ^{\circ 1}\left( v\right) :=f^{\circ 1}\left( v\right) \text{%
, }\forall v\in \mathcal{V}\text{, and }\left[ f\right] ^{\circ k}:=\left[
f^{\circ 1},\left[ f\right] ^{\circ \left( k-1\right) }\right] \text{ for }%
k\geq 2\text{.}  \label{inner Lie bracket of f with itself}
\end{equation}%
Then by definition, for any $k\geq 1$ and any $v\in \mathcal{V}^{\otimes k}$%
, $\left[ f\right] ^{\circ k}\left( v\right) $ is a first order differential
operator.

If we can prove that $f^{\circ k}\left( v^{k}\right) $ is a first order
differential operator for any $v^{k}$ in the form%
\begin{equation}
v^{k}=\left\{ 
\begin{array}{cc}
v_{1}, & \text{if }k=1 \\ 
\left[ v_{1},\dots ,\left[ v_{k-1},v_{k}\right] \right] , & \text{if }k\geq 2%
\end{array}%
\right. \text{, \ with }\left\{ v_{j}\right\} _{j=1}^{k}\subset \mathcal{V}%
\text{,}  \label{inner definition of v}
\end{equation}%
then since any $v\in \left[ \mathcal{V}\right] ^{k}$ can be approximated by
linear combinations of $v^{k}$ in the form of $\left( \ref{inner definition
of v}\right) $, by using that $f^{\circ k}:\mathcal{V}^{\otimes
k}\rightarrow \left( \mathcal{D}^{k}\left( \mathcal{U}\right) ,\left\vert
\cdot \right\vert _{k}\right) $ is linear and continuous (Definition \ref%
{Definition of f circ k}), and that $\left[ \mathcal{V}\right] ^{k}$ is a
closed subspace of $\mathcal{V}^{\otimes k}$, we can prove that $f^{\circ
k}\left( v\right) $ is a first order differential operator for any $v\in %
\left[ \mathcal{V}\right] ^{k}$.

We define linear map $\sigma :\left[ \mathcal{V}\right] ^{k}\rightarrow 
\mathcal{V}^{\otimes k}$ by assigning%
\begin{eqnarray}
\sigma \left( v_{1}\right)  &:&=v_{1}\text{, }\forall v_{1}\in \mathcal{V}%
\text{, if }k=1\text{,}  \label{inner definition of sigma v} \\
\sigma \left( \left[ v_{1},\dots ,\left[ v_{k-1},v_{k}\right] \right]
\right)  &:&=v_{1}\otimes \cdots \otimes v_{k-1}\otimes v_{k}\text{, }%
\forall \left\{ v_{j}\right\} _{j=1}^{k}\subset \mathcal{V}\text{, \ if }%
k\geq 2\text{.}  \notag
\end{eqnarray}%
For any $v^{k}$ in the form of $\left( \ref{inner definition of v}\right) $
with $\sigma $ defined at $\left( \ref{inner definition of sigma v}\right) $%
, we want to prove 
\begin{equation}
f^{\circ k}\left( v^{k}\right) \left( r\right) =\left[ f\right] ^{\circ
k}\left( \sigma \left( v^{k}\right) \right) \left( r\right) \text{, \ }%
\forall r\in C^{1}\left( \mathcal{U},\mathcal{U}\right) \text{.}
\label{inner relation 1}
\end{equation}%
By definition, $\left[ f\right] ^{\circ k}\left( \sigma \left( v^{k}\right)
\right) $ is a first order differential operator, and%
\begin{equation*}
\left[ f\right] ^{\circ k}\left( \sigma \left( v^{k}\right) \right) \left(
r\right) =\left( Dr\right) \left( \left[ f\right] ^{\circ k}\left( \sigma
\left( v^{k}\right) \right) \left( I_{d}\right) \right) \text{, \ }\forall
r\in C^{1}\left( \mathcal{U},\mathcal{U}\right) \text{.}
\end{equation*}%
If we can prove $\left( \ref{inner relation 1}\right) $, then 
\begin{equation*}
f^{\circ k}\left( v^{k}\right) \left( r\right) =\left( Dr\right) \left( 
\left[ f\right] ^{\circ k}\left( \sigma \left( v^{k}\right) \right) \left(
I_{d}\right) \right) =\left( Dr\right) \left( f^{\circ k}\left( v^{k}\right)
\left( I_{d}\right) \right) \text{, \ }\forall r\in C^{1}\left( \mathcal{U},%
\mathcal{U}\right) \text{.}
\end{equation*}%
Thus, in the following, we concentrate on proving $\left( \ref{inner
relation 1}\right) $.

It is clear that, $\left( \ref{inner relation 1}\right) $ is true when $k=1$%
. Indeed, for any $v^{1}\in \mathcal{V}$, since $\left[ f\right] ^{\circ
1}\left( v^{1}\right) :=f^{\circ 1}\left( v^{1}\right) $ (see $\left( \ref%
{inner Lie bracket of f with itself}\right) $) and $\sigma \left(
v^{1}\right) :=v^{1}$ (see $\left( \ref{inner definition of sigma v}\right) $%
), we have 
\begin{equation*}
\left[ f\right] ^{\circ 1}\left( v^{1}\right) =f^{\circ 1}\left(
v^{1}\right) =f^{\circ 1}\left( \sigma \left( v^{1}\right) \right) \text{, }%
\forall v^{1}\in \mathcal{V}\text{.}
\end{equation*}%
Then we use mathematical induction. Suppose that for integer $K\geq 1$, $%
k=1,2,\dots ,K$ and any $v^{k}$ in the form of $\left( \ref{inner definition
of v}\right) $, we have%
\begin{equation}
f^{\circ k}\left( v^{k}\right) =\left[ f\right] ^{\circ k}\left( \sigma
\left( v^{k}\right) \right) \text{.}  \label{inner inductive hypothesis}
\end{equation}%
We want prove that, for any $v_{0}\in \mathcal{V}$, and any $v^{K}$ in the
form of $\left( \ref{inner definition of v}\right) $,%
\begin{equation*}
f^{\circ \left( K+1\right) }\left( \left[ v_{0},v^{K}\right] \right) =\left[
f\right] ^{\circ \left( K+1\right) }\left( v_{0}\otimes \sigma \left(
v^{K}\right) \right) \text{.}
\end{equation*}

Based on definition $\left( \ref{inner definition of bracket of differential
operator}\right) $ and $\left( \ref{inner Lie bracket of f with itself}%
\right) $, we have, for any $r\in C^{1}\left( \mathcal{U},\mathcal{U}\right) 
$,%
\begin{eqnarray}
\left[ f\right] ^{\circ \left( K+1\right) }\left( v_{0}\otimes \sigma \left(
v^{K}\right) \right) \left( r\right)  &=&\left[ f^{\circ 1}\left(
v_{0}\right) ,\left[ f\right] ^{\circ K}\left( \sigma \left( v^{K}\right)
\right) \right] \left( r\right)   \label{inner 1} \\
&=&\left( Dr\right) \left( f^{\circ 1}\left( v_{0}\right) \circ \left[ f%
\right] ^{\circ K}\left( \sigma \left( v^{K}\right) \right) -\left[ f\right]
^{\circ K}\left( \sigma \left( v^{K}\right) \right) \circ f^{\circ 1}\left(
v_{0}\right) \right) \left( I_{d}\right) \text{.}  \notag
\end{eqnarray}%
If we assume in addition that $r\in C^{2}\left( \mathcal{U},\mathcal{U}%
\right) $, then by using that $f^{\circ 1}\left( v_{0}\right) $ and $\left[ f%
\right] ^{\circ K}\left( \sigma \left( v^{K}\right) \right) $ are first
order differential operators, we have 
\begin{eqnarray}
&&\left( Dr\right) \left( f^{\circ 1}\left( v_{0}\right) \circ \left[ f%
\right] ^{\circ K}\left( \sigma \left( v^{K}\right) \right) \right) \left(
I_{d}\right)   \label{inner 2} \\
&=&\left( f^{\circ 1}\left( v_{0}\right) \circ \left[ f\right] ^{\circ
K}\left( \sigma \left( v^{K}\right) \right) \right) \left( r\right) -\left(
D^{2}r\right) \left( \left[ f\right] ^{\circ K}\left( \sigma \left(
v^{K}\right) \right) \left( I_{d}\right) \right) \left( f^{\circ 1}\left(
v_{0}\right) \left( I_{d}\right) \right) \text{,}  \notag
\end{eqnarray}%
and%
\begin{eqnarray}
&&\left( Dr\right) \left( \left[ f\right] ^{\circ K}\left( \sigma \left(
v^{K}\right) \right) \circ f^{\circ 1}\left( v_{0}\right) \right) \left(
I_{d}\right)   \label{inner 3} \\
&=&\left( \left[ f\right] ^{\circ K}\left( \sigma \left( v^{K}\right)
\right) \circ f^{\circ 1}\left( v_{0}\right) \right) \left( r\right) -\left(
D^{2}r\right) \left( f^{\circ 1}\left( v_{0}\right) \left( I_{d}\right)
\right) \left( \left[ f\right] ^{\circ K}\left( \sigma \left( v^{K}\right)
\right) \left( I_{d}\right) \right) \text{.}  \notag
\end{eqnarray}%
Using inductive hypothesis $\left( \ref{inner inductive hypothesis}\right) $
and the definition of $f^{\circ \left( K+1\right) }$ in Definition \ref%
{Definition of f circ k}, we have%
\begin{eqnarray}
&&\left( f^{\circ 1}\left( v_{0}\right) \circ \left[ f\right] ^{\circ
K}\left( \sigma \left( v^{K}\right) \right) \right) -\left( \left[ f\right]
^{\circ K}\left( \sigma \left( v^{K}\right) \right) \circ f^{\circ 1}\left(
v_{0}\right) \right)   \label{inner 4} \\
&=&f^{\circ 1}\left( v_{0}\right) \circ f^{\circ K}\left( v^{K}\right)
-f^{\circ K}\left( v^{K}\right) \circ f^{\circ 1}\left( v_{0}\right)
=f^{\circ \left( K+1\right) }\left( v_{0}\otimes v^{K}-v^{K}\otimes
v_{0}\right) =f^{\circ \left( K+1\right) }\left( \left[ v_{0},v^{K}\right]
\right) \text{.}  \notag
\end{eqnarray}%
Since our differentiability is in Fr\'{e}chet's sense and $r\in C^{2}\left( 
\mathcal{U},\mathcal{U}\right) $, we have 
\begin{equation}
\left( D^{2}r\right) \left( u_{1}\otimes u_{2}\right) =\left( D^{2}r\right)
\left( u_{2}\otimes u_{1}\right) ,\forall u_{1}\in \mathcal{U},\forall
u_{2}\in \mathcal{U}.  \label{inner 5}
\end{equation}%
Thus, combining $\left( \ref{inner 1}\right) $, $\left( \ref{inner 2}\right) 
$, $\left( \ref{inner 3}\right) $, $\left( \ref{inner 4}\right) $ and $%
\left( \ref{inner 5}\right) $, we have, 
\begin{equation}
\left[ f\right] ^{\circ \left( K+1\right) }\left( v_{0}\otimes \sigma \left(
v^{K}\right) \right) \left( r\right) =f^{\circ \left( K+1\right) }\left( %
\left[ v_{0},v^{K}\right] \right) \left( r\right) \text{, \ }\forall r\in
C^{2}\left( \mathcal{U},\mathcal{U}\right) \text{.}
\label{inner what we want to prove}
\end{equation}%
Since $\left[ f\right] ^{\circ \left( K+1\right) }\left( v_{0}\otimes \sigma
\left( v^{K}\right) \right) :=\left[ f^{\circ 1}\left( v_{0}\right) ,\left[ f%
\right] ^{\circ K}\left( \sigma \left( v^{K}\right) \right) \right] $ has
the explicit form $\left( \ref{inner definition of bracket of differential
operator}\right) $, by comparing the "coefficients" of $\left\{
D^{j}r\right\} _{j=0}^{K+1}$ for any $r\in C^{K+1}\left( \mathcal{U},%
\mathcal{U}\right) $, we get that $\left( \ref{inner what we want to prove}%
\right) $ holds for any $r\in C^{1}\left( \mathcal{U},\mathcal{U}\right) $.
\end{proof}

\begin{lemma}
\label{Lemma Euler expansion of the solution of ODE}Suppose $\mathcal{V}$
and $\mathcal{U}$ are two Banach spaces, $f\in L\left( \mathcal{V},C^{\gamma
}\left( \mathcal{U},\mathcal{U}\right) \right) $ for some $\gamma >1$ and $%
\xi \in \mathcal{U}$. Denote $\lfloor \gamma \rfloor $ as the largest
integer which is strictly less than $\gamma $. Suppose $g\in G^{\lfloor
\gamma \rfloor +1}\left( \mathcal{V}\right) $. Then, there exists constant $%
C_{\gamma }$, which only depends on $\gamma $, such that, the unique
solution of the ordinary differential equation%
\begin{eqnarray}
dy_{u} &=&\sum_{k=1}^{\lfloor \gamma \rfloor }f^{\circ k}\pi _{k}\left( \log
_{\lfloor \gamma \rfloor +1}g\right) \left( I_{d}\right) \left( y_{u}\right)
du\text{, }u\in \left[ 0,1\right] \text{,}  \label{first ode} \\
\text{ }y_{0} &=&\xi +f^{\circ \left( \lfloor \gamma \rfloor +1\right) }\pi
_{\lfloor \gamma \rfloor +1}\left( \log _{\lfloor \gamma \rfloor +1}\left(
g\right) \right) \left( I_{d}\right) \left( \xi \right) \text{,}  \notag
\end{eqnarray}%
satisfies%
\begin{equation*}
\left\Vert y_{1}-\xi -\sum_{k=1}^{\lfloor \gamma \rfloor +1}f^{\circ k}\pi
_{k}\left( g\right) \left( I_{d}\right) \left( \xi \right) \right\Vert \leq
C_{\gamma }\left( \left\vert f\right\vert _{Lip\left( \gamma \right)
}\left\Vert g\right\Vert \right) ^{\gamma +1}\text{.}
\end{equation*}
\end{lemma}

\begin{proof}
\label{Proof of Lemma Euler expansion of the solution of ODE}Denote $\left\{
\gamma \right\} :=\gamma -\lfloor \gamma \rfloor $, and denote $N:=\lfloor
\gamma \rfloor +1$. We assume $\left\vert f\right\vert _{Lip\left( \gamma
\right) }=1$. Otherwise, we replace $f$ by $\left\vert f\right\vert
_{Lip\left( \gamma \right) }^{-1}f$ and replace $g$ by $\delta _{\left\vert
f\right\vert _{Lip\left( \gamma \right) }}g$ (with $\delta _{\lambda
}g:=1+\sum_{k=1}^{N}\lambda ^{k}\pi _{k}\left( g\right) $). For integer $%
k=1,2,\dots ,\lfloor \gamma \rfloor +1$ and $v\in \mathcal{V}^{\otimes k}$,
based on Definition of differential operator $f^{\circ k}\left( v\right) $
in Definition \ref{Definition of f circ k} on page \pageref{Definition of f
circ k}, $f^{\circ k}\left( v\right) \left( I_{d}\right) \in C^{\gamma
-k+1}\left( \mathcal{U},\mathcal{U}\right) $ and\ 
\begin{equation}
\left\vert f^{\circ k}\left( v\right) \left( I_{d}\right) \right\vert
_{Lip\left( \gamma -k+1\right) }=\left\Vert v\right\Vert \left\vert
f\right\vert _{Lip\left( \gamma \right) }^{k}\left\vert \left( \frac{f}{%
\left\vert f\right\vert _{Lip\left( \gamma \right) }}\right) ^{\circ
k}\left( \frac{v}{\left\Vert v\right\Vert }\right) \left( I_{d}\right)
\right\vert _{Lip\left( \gamma -k+1\right) }\leq C_{\gamma }\left\Vert
v\right\Vert \left\vert f\right\vert _{Lip\left( \gamma \right) }^{k}\text{.}
\label{inner estimate of fcirck Id}
\end{equation}

When $\left\Vert g\right\Vert >1$, it can be computed that ($\left\vert
f\right\vert _{Lip\left( \gamma \right) }=1$) 
\begin{equation*}
\left\Vert \sum_{k=1}^{N-1}f^{\circ k}\left( I_{d}\right) \left( \xi \right)
\pi _{k}\left( g\right) \right\Vert \leq C_{\gamma }\left\Vert g\right\Vert
^{N-1}\text{\ and }\left\Vert y_{1}-\xi \right\Vert \leq \left\Vert
y_{0}-\xi \right\Vert +\left\Vert y\right\Vert _{1-var,\left[ 0,1\right]
}\leq C_{\gamma }\left\Vert g\right\Vert ^{N}\text{.}
\end{equation*}%
Thus, ($\left\Vert g\right\Vert >1$, $N\leq \gamma +1$)%
\begin{equation*}
\left\Vert y_{1}-\xi -\sum_{k=1}^{N}f^{\circ k}\left( I_{d}\right) \left(
\xi \right) \pi _{k}\left( g\right) \right\Vert \leq C_{\gamma }\left\Vert
g\right\Vert ^{N}\leq C_{\gamma }\left\Vert g\right\Vert ^{\gamma +1}\text{,}
\end{equation*}%
and lemma holds. In the following, we assume $\left\vert f\right\vert
_{Lip\left( \gamma \right) }=1$ and $\left\Vert g\right\Vert \leq 1$.

It is clear that, the solution $y$ of the ordinary differential equation $%
\left( \ref{first ode}\right) $ satisfies (since $\left\vert f\right\vert
_{Lip\left( \gamma \right) }=1$ and $\left\Vert g\right\Vert \leq 1$)%
\begin{equation}
\sup_{u\in \left[ 0,1\right] }\left\Vert y_{u}-\xi \right\Vert \leq
\left\Vert y_{0}-\xi \right\Vert +\left\Vert y\right\Vert _{1-var,\left[ 0,1%
\right] }\leq C_{\gamma }\left\Vert g\right\Vert \text{.}
\label{inner estimation of ODE solution}
\end{equation}

For integer $k=1,\dots ,N$, denote differential operator $F^{k}\ $as%
\begin{equation*}
F^{k}:=f^{\circ k}\pi _{k}\left( \log _{N}\left( g\right) \right) \text{.}
\end{equation*}%
Based on Lemma \ref{Lemma exchange of Id and dot}, $\left\{ F^{k}\right\}
_{k=1}^{N}$ are first order differential operators, and satisfies%
\begin{equation*}
F^{k}\left( r\right) =\left( Dr\right) F^{k}\left( I_{d}\right) \text{, \ }%
\forall r\in C^{1}\left( \mathcal{U},\mathcal{U}\right) \text{.}
\end{equation*}%
Similar as $\left( \ref{inner estimate of fcirck Id}\right) $, since we
assumed $\left\vert f\right\vert _{Lip\left( \gamma \right) }=1$, for $1\leq
k\leq N-1$, we have%
\begin{equation}
\left\vert DF^{k}\left( I_{d}\right) \right\vert _{Lip\left( \gamma
-k\right) }\leq C_{\gamma }\left\Vert g\right\Vert ^{k}\text{;}
\label{inner bound of functional of f}
\end{equation}%
for $k_{i}\geq 1$, $\sum_{i=1}^{j}k_{i}=k\leq N$, we have 
\begin{equation}
\left\vert \left( F^{k_{j}}\circ \cdots \circ F^{k_{1}}\right) \left(
I_{d}\right) \right\vert _{Lip\left( \gamma +1-k\right) }\leq C_{\gamma
}\left\Vert g\right\Vert ^{k}\text{.}
\label{inner composition of differential operators}
\end{equation}

By using the fact that $y$ satisfies $\left( \ref{first ode}\right) $, we
have%
\begin{eqnarray}
&&y_{1}-\xi -f^{\circ N}\pi _{N}\left( \log _{N}\left( g\right) \right)
\left( I_{d}\right) \left( \xi \right) -\sum_{k=1}^{N-1}F^{k}\left(
I_{d}\right) \left( \xi \right)  \label{inner ode as a whole} \\
&=&\sum_{k=1}^{N-1}\iint_{0\leq u_{1}\leq u_{2}\leq 1}\left( F^{k}\left(
I_{d}\right) \left( y_{0}\right) -F^{k}\left( I_{d}\right) \left( \xi
\right) \right) du_{1}du_{2}  \notag \\
&&+\sum_{1\leq k_{i}\leq N-1,i=1,2}\iint_{0\leq u_{1}\leq u_{2}\leq
1}DF^{k_{1}}\left( I_{d}\right) F^{k_{2}}\left( I_{d}\right) \left(
y_{u_{1}}\right) du_{1}du_{2}\text{.}  \notag
\end{eqnarray}%
Since $y_{0}=\xi +f^{\circ N}\pi _{N}\left( \log _{N}g\right) \left(
I_{d}\right) \left( \xi \right) $, by using $\left( \ref{inner composition
of differential operators}\right) $, we have ($\left\vert f\right\vert
_{Lip\left( \gamma \right) }=1$, $\left\Vert g\right\Vert \leq 1$ and $%
\gamma \leq N$)%
\begin{equation*}
\left\Vert \sum_{k=1}^{N-1}\iint_{0\leq u_{1}\leq u_{2}\leq 1}F^{k}\left(
I_{d}\right) \left( y_{0}\right) -F^{k}\left( I_{d}\right) \left( \xi
\right) du_{1}du_{2}\right\Vert \leq C_{\gamma }\left\Vert g\right\Vert
^{1+N}\leq C_{\gamma }\left\Vert g\right\Vert ^{1+\gamma }\text{.}
\end{equation*}%
When $k_{1}\geq 1$, $k_{2}\geq 1$, $k_{1}+k_{2}\leq N$, using that $%
F^{k_{2}} $ is a first order differential operator, we have $%
DF^{k_{1}}\left( I_{d}\right) F^{k_{2}}\left( I_{d}\right) =\left(
F^{k_{2}}\circ F^{k_{1}}\right) \left( I_{d}\right) $. Thus, 
\begin{equation}
\iint_{0\leq u_{1}\leq u_{2}\leq 1}DF^{k_{1}}\left( I_{d}\right)
F^{k_{2}}\left( I_{d}\right) \left( y_{u_{1}}\right)
du_{1}du_{2}=\iint_{0\leq u_{1}\leq u_{2}\leq 1}\left( F^{k_{2}}\circ
F^{k_{1}}\right) \left( I_{d}\right) \left( y_{u_{1}}\right) du_{1}du_{2}%
\text{.}  \label{inner ode part 1}
\end{equation}%
When $1\leq k_{1}\leq N-1$, $1\leq k_{2}\leq N-1$, $k_{1}+k_{2}\geq N+1$, by
combining $\left( \ref{inner bound of functional of f}\right) $ and $\left( %
\ref{inner composition of differential operators}\right) $, we get 
\begin{equation}
\left\Vert \iint_{0\leq u_{1}\leq u_{2}\leq 1}DF^{k_{2}}\left( I_{d}\right)
F^{k_{1}}\left( I_{d}\right) \left( y_{u_{1}}\right) du_{1}du_{2}\right\Vert
\leq C_{\gamma }\left\Vert g\right\Vert ^{k_{1}+k_{2}}\leq C_{\gamma
}\left\Vert g\right\Vert ^{N+1}\leq C_{\gamma }\left\Vert g\right\Vert
^{\gamma +1}\text{.}  \label{inner ode part 2}
\end{equation}%
Therefore, combining $\left( \ref{inner ode as a whole}\right) $, $\left( %
\ref{inner ode part 1}\right) $ and $\left( \ref{inner ode part 2}\right) $,
we have%
\begin{eqnarray*}
&&\left\Vert y_{1}-\xi -\sum_{k=1}^{N}F^{k}\left( I_{d}\right) \left( \xi
\right) -\sum_{1\leq k_{i}\leq N-1,k_{1}+k_{2}\leq N}\iint_{0\leq u_{1}\leq
u_{2}\leq 1}\left( F^{k_{2}}\circ F^{k_{1}}\right) \left( I_{d}\right)
\left( y_{u_{1}}\right) du_{1}du_{2}\right\Vert \\
&\leq &C_{\gamma }\left\Vert g\right\Vert ^{\gamma +1}\text{.}
\end{eqnarray*}

Then we continue to estimate%
\begin{equation*}
\sum_{1\leq k_{i}\leq N-1,k_{1}+k_{2}\leq N}\iint_{0\leq u_{1}\leq u_{2}\leq
1}\left( F^{k_{2}}\circ F^{k_{1}}\right) \left( I_{d}\right) \left(
y_{u_{1}}\right) du_{1}du_{2}\text{.}
\end{equation*}%
When $1\leq k_{1}\leq N-1$, $1\leq k_{2}\leq N-1$, $k_{1}+k_{2}=N$, by using 
$\left( \ref{inner composition of differential operators}\right) \,$\ and $%
\left( \ref{inner estimation of ODE solution}\right) $, we have%
\begin{eqnarray*}
&&\left\Vert \iint_{0\leq u_{1}\leq u_{2}\leq 1}\left( \left( F^{k_{2}}\circ
F^{k_{1}}\right) \left( I_{d}\right) \left( y_{u_{1}}\right) -\left(
F^{k_{2}}\circ F^{k_{1}}\right) \left( I_{d}\right) \left( \xi \right)
\right) du_{1}du_{2}\right\Vert \\
&\leq &C_{\gamma }\left\Vert g\right\Vert ^{N}\sup_{u\in \left[ 0,1\right]
}\left\Vert y_{u}-\xi \right\Vert ^{\left\{ \gamma \right\} }\leq C_{\gamma
}\left\Vert g\right\Vert ^{\gamma +1}\text{.}
\end{eqnarray*}%
When $k_{1}\geq 1$, $k_{2}\geq 1$ and $k_{1}+k_{2}\leq N-1$, we have%
\begin{eqnarray*}
&&\sum_{k_{i}\geq 1,k_{1}+k_{2}\leq N-1}\iint_{0\leq u_{1}\leq u_{2}\leq
1}\left( \left( F^{k_{2}}\circ F^{k_{1}}\right) \left( I_{d}\right) \left(
y_{u_{1}}\right) -\left( F^{k_{2}}\circ F^{k_{1}}\right) \left( I_{d}\right)
\left( \xi \right) \right) du_{1}du_{2} \\
&=&\sum_{k_{i}\geq 1,k_{1}+k_{2}\leq N-1}\iint_{0\leq u_{1}\leq u_{2}\leq
1}\left( \left( F^{k_{2}}\circ F^{k_{1}}\right) \left( I_{d}\right) \left(
y_{0}\right) -\left( F^{k_{2}}\circ F^{k_{1}}\right) \left( I_{d}\right)
\left( \xi \right) \right) du_{1}du_{2} \\
&&+\sum_{k_{i}\geq 1,k_{1}+k_{2}\leq N-1,k_{3}\leq
N-1}\iiint_{0<u_{1}<u_{2}<u_{3}}D\left( F^{k_{2}}\circ F^{k_{1}}\right)
\left( I_{d}\right) F^{k_{3}}\left( I_{d}\right) \left( y_{u_{1}}\right)
du_{1}du_{2}du_{3}\text{.}
\end{eqnarray*}%
Then since $y_{0}=\xi +f^{\circ N}\pi _{N}\left( \log _{N}g\right) \left(
I_{d}\right) \left( \xi \right) $, by using $\left( \ref{inner composition
of differential operators}\right) $ and $\left( \ref{inner estimation of ODE
solution}\right) $, we have%
\begin{eqnarray*}
&&\left\Vert \sum_{k_{i}\geq 1,k_{1}+k_{2}\leq N}\iint_{0\leq u_{1}\leq
u_{2}\leq 1}\left( \left( F^{k_{2}}\circ F^{k_{1}}\right) \left(
I_{d}\right) \left( y_{0}\right) -\left( F^{k_{2}}\circ F^{k_{1}}\right)
\left( I_{d}\right) \left( \xi \right) \right) du_{1}du_{2}\right\Vert \\
&\leq &C_{\gamma }\left\Vert g\right\Vert ^{2}\left\Vert y_{0}-\xi
\right\Vert \leq C_{\gamma }\left\Vert g\right\Vert ^{\gamma +1}\text{.}
\end{eqnarray*}%
Then similar as in $\left( \ref{inner ode part 1}\right) $ and $\left( \ref%
{inner ode part 2}\right) $, when $k_{1}+k_{2}+k_{3}\leq N$, we have 
\begin{eqnarray*}
&&\iiint_{0<u_{1}<u_{2}<u_{3}}D\left( F^{k_{2}}\circ F^{k_{1}}\right) \left(
I_{d}\right) F^{k_{3}}\left( I_{d}\right) \left( y_{u_{1}}\right)
du_{1}du_{2}du_{3} \\
&=&\iiint_{0<u_{1}<u_{2}<u_{3}}\left( F^{k_{3}}\circ F^{k_{2}}\circ
F^{k_{1}}\right) \left( I_{d}\right) \left( y_{u_{1}}\right)
du_{1}du_{2}du_{3}\text{;}
\end{eqnarray*}%
when $k_{1}+k_{2}+k_{3}\geq N+1$, we have 
\begin{equation*}
\iiint_{0<u_{1}<u_{2}<u_{3}}D\left( F^{k_{2}}\circ F^{k_{1}}\right) \left(
I_{d}\right) F^{k_{3}}\left( I_{d}\right) \left( y_{u_{1}}\right)
du_{1}du_{2}du_{3}\leq C_{\gamma }\left\Vert g\right\Vert ^{\gamma +1}.
\end{equation*}

Repeating this "subtraction and estimation" process for $N$ times, we get 
\begin{equation*}
\left\Vert y_{1}-\xi -\sum_{j=1}^{N}\frac{1}{j!}\sum_{k_{i}\geq
1,k_{1}+\cdots +k_{j}\leq N}\left( F^{k_{j}}\circ \cdots \circ
F^{k_{1}}\right) \left( I_{d}\right) \left( \xi \right) \right\Vert \leq
C_{\gamma }\left\Vert g\right\Vert ^{\gamma +1}\text{.}
\end{equation*}%
Since $f^{\circ k}$ is linear in $\mathcal{V}^{\otimes k}$ (Definition \ref%
{Definition of f circ k} on page \pageref{Definition of f circ k}), we have 
\begin{eqnarray*}
&&\sum_{j=1}^{N}\frac{1}{j!}\sum_{k_{i}\geq 1,k_{1}+\cdots +k_{j}\leq
N}\left( F^{k_{j}}\circ \cdots \circ F^{k_{1}}\right) \left( I_{d}\right)
\left( \xi \right) \\
&=&\sum_{j=1}^{N}\frac{1}{j!}\sum_{k_{i}\geq 1,k_{1}+\cdots +k_{j}\leq
N}f^{\circ \left( k_{1}+\cdots +k_{j}\right) }\pi _{k_{j}}\left( \log
_{N}\left( g\right) \right) \otimes \cdots \otimes \pi _{k_{1}}\left( \log
_{N}\left( g\right) \right) \left( I_{d}\right) \left( \xi \right) \\
&=&\sum_{k=1}^{N}f^{\circ k}\left( \sum_{k_{i}\geq 1,k_{1}+\cdots +k_{j}=k}%
\frac{1}{j!}\pi _{k_{j}}\left( \log _{N}\left( g\right) \right) \otimes
\cdots \otimes \pi _{k_{1}}\left( \log _{N}\left( g\right) \right) \right)
\left( I_{d}\right) \left( \xi \right) \\
&=&\sum_{k=1}^{N}f^{\circ k}\pi _{k}\left( g\right) \left( I_{d}\right)
\left( \xi \right) \text{.}
\end{eqnarray*}%
Therefore, we have 
\begin{equation*}
\left\Vert y_{1}-\xi -\left( \sum_{k=1}^{N}f^{\circ k}\pi _{k}\left(
g\right) \left( I_{d}\right) \left( \xi \right) \right) \right\Vert \leq
C_{\gamma }\left\Vert g\right\Vert ^{\gamma +1}\text{.}
\end{equation*}
\end{proof}

\begin{lemma}
\label{Lemma taylor expansion of f(y)}Suppose $\mathcal{V}$ and $\mathcal{U}$
are two Banach spaces, $f\in L\left( \mathcal{V},C^{\gamma }\left( \mathcal{U%
},\mathcal{U}\right) \right) $ for some $\gamma >1$ and $\xi \in \mathcal{U}$%
. Suppose $g\in G^{\lfloor \gamma \rfloor +1}\left( \mathcal{V}\right) $.
Then, the unique solution of the ordinary differential equation%
\begin{eqnarray*}
dy_{u} &=&\sum_{k=1}^{\lfloor \gamma \rfloor }f^{\circ k}\pi _{k}\left( \log
_{\lfloor \gamma \rfloor +1}\left( g\right) \right) \left( I_{d}\right)
\left( y_{u}\right) du\text{, }u\in \left[ 0,1\right] \text{,} \\
\text{ }y_{0} &=&\xi +f^{\circ \left( \lfloor \gamma \rfloor +1\right) }\pi
_{\lfloor \gamma \rfloor +1}\left( \log _{\lfloor \gamma \rfloor +1}\left(
g\right) \right) \left( I_{d}\right) \left( \xi \right) \text{,}
\end{eqnarray*}%
satisfies, for $k=1,2,\dots ,\lfloor \gamma \rfloor +1$, and any $v\in 
\mathcal{V}^{\otimes k}$,%
\begin{equation}
\left\Vert f^{\circ k}\left( v\right) \left( I_{d}\right) \left(
y_{1}\right) -\sum_{j=0}^{\lfloor \gamma \rfloor +1-k}f^{\circ \left(
j+k\right) }\left( \pi _{j}\left( g\right) \otimes v\right) \left(
I_{d}\right) \left( \xi \right) \right\Vert \leq C_{\gamma }\left\Vert
v\right\Vert \left\vert f\right\vert _{Lip\left( \gamma \right) }^{\gamma
+1}\left\Vert g\right\Vert ^{\gamma +1-k}\text{.}
\label{inner taylor expansiono f f(y)}
\end{equation}
\end{lemma}

\begin{proof}
This lemma can be proved similarly as Lemma \ref{Lemma Euler expansion of
the solution of ODE}.
\end{proof}

\begin{lemma}
\label{Lemma difference of two steps and one step}Suppose $\mathcal{V}$ and $%
\mathcal{U}$ are two Banach spaces, $f\in L\left( \mathcal{V},C^{\gamma
}\left( \mathcal{U},\mathcal{U}\right) \right) $ for some $\gamma >1$ and $%
\xi \in \mathcal{U}$. Denote $N:=\lfloor \gamma \rfloor +1$, and suppose $%
g\in G^{N}\left( \mathcal{V}\right) $. Denote $y^{g}$ as the solution of the
ordinary differential equation:%
\begin{eqnarray*}
dy_{u}^{g} &=&\left( \sum_{k=1}^{N-1}f^{\circ k}\pi _{k}\left( \log
_{N}\left( g\right) \right) \left( I_{d}\right) \left( y_{u}^{g}\right)
\right) du\text{, }u\in \left[ 0,1\right] \text{,} \\
\text{ }y_{0}^{g} &=&\xi +f^{\circ N}\pi _{N}\left( \log _{N}\left( g\right)
\right) \left( I_{d}\right) \left( \xi \right) \text{.}
\end{eqnarray*}%
For $g,h\in G^{N}\left( \mathcal{V}\right) $, denote $y^{g,h}$ as the unique
solution to the integral equation:%
\begin{equation*}
y_{t}^{g,h}=\left\{ 
\begin{array}{cc}
\xi +f^{\circ N}\pi _{N}\left( \log _{N}\left( g\right) \right) \left(
I_{d}\right) \left( \xi \right) +\int_{0}^{t}\left( \sum_{k=1}^{N-1}f^{\circ
k}\pi _{k}\left( \log _{N}\left( g\right) \right) \left( I_{d}\right) \left(
y_{u}^{g,h}\right) \right) du\text{,} & t\in \left[ 0,1\right] \\ 
y_{1}^{g,h}+f^{\circ N}\pi _{N}\left( \log _{N}\left( h\right) \right)
\left( I_{d}\right) \left( y_{1}^{g,h}\right) +\int_{1}^{t}\left(
\sum_{k=1}^{N-1}f^{\circ k}\pi _{k}\left( \log _{N}\left( h\right) \right)
\left( I_{d}\right) \left( y_{u}^{g,h}\right) \right) du\text{,} & t\in (1,2]%
\end{array}%
\right. \text{.}
\end{equation*}%
Then we have%
\begin{equation*}
\left\Vert y_{2}^{g,h}-y_{1}^{g\otimes h}\right\Vert \leq C_{\gamma
}\left\vert f\right\vert _{Lip\left( \gamma \right) }^{\gamma +1}\left(
\left\Vert g\right\Vert \vee \left\Vert h\right\Vert \vee \left\Vert
g\otimes h\right\Vert \right) ^{\gamma +1}\text{.}
\end{equation*}
\end{lemma}

\begin{proof}
We only prove the Lemma when $\left\vert f\right\vert _{Lip\left( \gamma
\right) }=1$. Otherwise, we replace $f$ by $\left\vert f\right\vert
_{Lip\left( \gamma \right) }^{-1}f$, and replace $g$ and $h$ by $\delta
_{\left\vert f\right\vert _{Lip\left( \gamma \right) }}g$ and $\delta
_{\left\vert f\right\vert _{Lip\left( \gamma \right) }}h$ respectively.

Since $\sum_{k=1}^{N-1}f^{\circ k}\pi _{k}\left( \log _{N}\left( g\right)
\right) \left( I_{d}\right) \in C^{1}\left( \mathcal{U},\mathcal{U}\right) $%
, based on the definition of $y^{g,h}$ and $y^{g}$, we have $%
y_{t}^{g,h}=y_{t}^{g}$, $t\in \left[ 0,1\right] $. For $g,h\in G^{N}\left( 
\mathcal{V}\right) $, by using Lemma \ref{Lemma Euler expansion of the
solution of ODE}, we get%
\begin{eqnarray*}
&&\left\Vert y_{2}^{g,h}-y_{1}^{g\otimes h}\right\Vert \\
&=&\left\Vert y_{1}^{g}-\xi +y_{2}^{g,h}-y_{1}^{g}-\left( y_{1}^{g\otimes
h}-\xi \right) \right\Vert \\
&\leq &\left\Vert \sum_{k=1}^{N}f^{\circ k}\pi _{k}\left( g\right) \left(
I_{d}\right) \left( \xi \right) +\sum_{k=1}^{N}f^{\circ k}\pi _{k}\left(
h\right) \left( I_{d}\right) \left( y_{1}^{g}\right) -\sum_{k=1}^{N}f^{\circ
k}\pi _{k}\left( g\otimes h\right) \left( I_{d}\right) \left( \xi \right)
\right\Vert \\
&&+C_{\gamma }\left( \left\Vert g\right\Vert \vee \left\Vert h\right\Vert
\vee \left\Vert g\otimes h\right\Vert \right) ^{\gamma +1}\text{.}
\end{eqnarray*}%
Based on Lemma \ref{Lemma taylor expansion of f(y)}, for $k=1,2,\dots ,N$,%
\begin{equation*}
\left\Vert f^{\circ k}\pi _{k}\left( h\right) \left( I_{d}\right) \left(
y_{1}^{g}\right) -\sum_{j=0}^{N-k}f^{\circ \left( j+k\right) }\left( \pi
_{j}\left( g\right) \otimes \pi _{k}\left( h\right) \right) \left(
I_{d}\right) \left( \xi \right) \right\Vert \leq C_{\gamma }\left\Vert
h\right\Vert ^{k}\left\Vert g\right\Vert ^{\gamma +1-k}\text{.}
\end{equation*}%
As a result,%
\begin{eqnarray*}
&&\left\Vert y_{2}^{g,h}-y_{1}^{g\otimes h}\right\Vert \\
&\leq &\left\Vert \sum_{k=1}^{N}f^{\circ k}\pi _{k}\left( g\right) \left(
I_{d}\right) \left( \xi \right) +\sum_{k=1}^{N}\sum_{j=0}^{N-k}f^{\circ
\left( j+k\right) }\left( \pi _{j}\left( g\right) \otimes \pi _{k}\left(
h\right) \right) \left( I_{d}\right) \left( \xi \right)
-\sum_{k=1}^{N}f^{\circ k}\pi _{k}\left( g\otimes h\right) \left(
I_{d}\right) \left( \xi \right) \right\Vert \\
&&+\,C_{\gamma }\left( \left\Vert g\right\Vert \vee \left\Vert h\right\Vert
\vee \left\Vert g\otimes h\right\Vert \right) ^{\gamma +1}+C_{\gamma
}\sum_{k=1}^{N}\left\Vert h\right\Vert ^{k}\left\Vert g\right\Vert ^{\gamma
+1-k} \\
&\leq &C_{\gamma }\left( \left\Vert g\right\Vert \vee \left\Vert
h\right\Vert \vee \left\Vert g\otimes h\right\Vert \right) ^{\gamma +1}\text{%
.}
\end{eqnarray*}
\end{proof}

\begin{lemma}
\label{Lemma uniform bound in p-var of solution of ode}Suppose $\mathcal{U}$
and $\mathcal{V}$ are two Banach spaces, $x:\left[ 0,T\right] \rightarrow 
\mathcal{V}$ is a continuous bounded variation path, and $f\in L\left( 
\mathcal{V},C^{\gamma }\left( \mathcal{U},\mathcal{U}\right) \right) $ for $%
\gamma \geq 1$. Denote $y:\left[ 0,T\right] \rightarrow \mathcal{U}$ as the
unique solution of the ordinary differential equation%
\begin{equation}
dy=f\left( y\right) dx\text{, }y_{0}=\xi \in \mathcal{U}\text{.}
\label{ODE in lemma}
\end{equation}%
Then for any $p\in \lbrack 1,\gamma +1)$, there exists constant $C_{p,\gamma
}$ (which only depends on $p$ and $\gamma $), such that for any interval $%
\left[ s,t\right] \subset \left[ 0,T\right] $ satisfying $\left\vert
f\right\vert _{Lip\left( \gamma \right) }\left\Vert S_{\left[ p\right]
}\left( x\right) \right\Vert _{p-var,\left[ s,t\right] }\leq 1$, we have%
\begin{equation}
\left\Vert S_{\left[ p\right] }\left( y\right) \right\Vert _{p-var,\left[ s,t%
\right] }\leq C_{p,\gamma }\left\vert f\right\vert _{Lip\left( \gamma
\right) }\left\Vert S_{\left[ p\right] }\left( x\right) \right\Vert _{p-var,%
\left[ s,t\right] }\text{.}  \label{uniform p-variation estimate}
\end{equation}
\end{lemma}

\begin{proof}
Define $h:\mathcal{V}\oplus \mathcal{U\rightarrow }L\left( \mathcal{V}\oplus 
\mathcal{U},\mathcal{V}\oplus \mathcal{U}\right) $ as%
\begin{equation*}
h\left( v_{1},u_{1}\right) \left( v_{2},u_{2}\right) =\left( v_{2},f\left(
v_{2}\right) \left( u_{1}+\xi \right) \right) \text{, \ }\forall
v_{1},v_{2}\in \mathcal{V}\text{, }\forall u_{1},u_{2}\in \mathcal{U}\text{.}
\end{equation*}%
We define geometric $p$-rough paths $Z\left( n\right) :\left[ 0,T\right]
\rightarrow G^{\left[ p\right] }\left( \mathcal{V}\oplus \mathcal{U}\right) $%
, $n\geq 0$, recursively as the rough integral (in the sense of Def 4.9 \cite%
{Lyonsnotes}):%
\begin{eqnarray}
Z\left( 0\right) _{t} &:&=\left( S_{\left[ p\right] }\left( x\right)
_{t},0\right) \in G^{\left[ p\right] }\left( \mathcal{V}\oplus \mathcal{U}%
\right) \text{, }t\in \left[ 0,T\right] \text{,}
\label{inner definition of Y(n+1)} \\
Z\left( n+1\right) _{t} &:&=\int_{0}^{t}h\left( Z\left( n\right) \right)
dZ\left( n\right) \text{, \ }t\in \left[ 0,T\right] \text{, }n\geq 0\text{,}
\notag
\end{eqnarray}%
and define $Y\left( n\right) :\left[ 0,T\right] \rightarrow G^{\left[ p%
\right] }\left( \mathcal{U}\right) $ as $Y\left( n\right) :=\pi _{G^{\left[ p%
\right] }\left( \mathcal{U}\right) }Z\left( n\right) $. Then based on Prop
5.9 \cite{Lyonsnotes}, there exists constant $C_{p,\gamma }$, which only
depends on $p$ and $\gamma $ and is finite whenever $\gamma >p-1$, such
that, for any interval $\left[ s,t\right] $ satisfying $\left\vert
f\right\vert _{Lip\left( \gamma \right) }\left\Vert S_{\left[ p\right]
}\left( x\right) \right\Vert _{p-var,\left[ s,t\right] }\leq 1$, we have%
\begin{equation}
\sup_{n}\left\Vert Y\left( n\right) \right\Vert _{p-var,\left[ s,t\right]
}\leq C_{p,\gamma }\left\vert f\right\vert _{Lip\left( \gamma \right)
}\left\Vert S_{\left[ p\right] }\left( x\right) \right\Vert _{p-var,\left[
s,t\right] }\text{.}  \label{inner estimation of p-var of Picard iteration}
\end{equation}%
(Indeed, by properly scaling $f$ and $S_{\left[ p\right] }\left( x\right) $,
the constant $C_{p,\gamma }$ in $\left( \ref{inner estimation of p-var of
Picard iteration}\right) $ can be chosen to be independent of $\left\vert
f\right\vert _{Lip\left( \gamma \right) }$ and $\left\Vert S_{\left[ p\right]
}\left( x\right) \right\Vert _{p-var,\left[ 0,T\right] }$.) On the other
hand, since $x$ is continuous with bounded variation, it can be checked
that, if we define continuous bounded variation paths $y\left( n\right) :%
\left[ 0,T\right] \rightarrow \mathcal{U}$, $n\geq 1$, recursively as%
\begin{eqnarray}
y\left( 0\right) _{t} &\equiv &0\in \mathcal{U}\text{, }t\in \left[ 0,T%
\right] \text{,}  \label{inner iteration of solution of ode} \\
y\left( n+1\right) _{t} &=&\int_{0}^{t}f\left( y\left( n\right) +\xi \right)
dx\text{, }t\in \left[ 0,T\right] \text{,}  \notag
\end{eqnarray}%
then based on the definition of rough integral in Def 4.9 \cite{Lyonsnotes},
it can be checked that, 
\begin{equation}
Y\left( n\right) =S_{\left[ p\right] }\left( y\left( n\right) \right) \text{%
, }\forall n\geq 0\text{.}  \label{inner relation between Y(n) and y(n)}
\end{equation}%
Combined with $\left( \ref{inner estimation of p-var of Picard iteration}%
\right) $, for interval $\left[ s,t\right] $ satisfying $\left\vert
f\right\vert _{Lip\left( \gamma \right) }\left\Vert S_{\left[ p\right]
}\left( x\right) \right\Vert _{p-var,\left[ s,t\right] }\leq 1$, we have%
\begin{equation}
\sup_{n}\left\Vert S_{\left[ p\right] }\left( y\left( n\right) \right)
\right\Vert _{p-var,\left[ s,t\right] }\leq C_{p,\gamma }\left\vert
f\right\vert _{Lip\left( \gamma \right) }\left\Vert S_{\left[ p\right]
}\left( x\right) \right\Vert _{p-var,\left[ s,t\right] }\text{.}
\label{inner uniform bound of p-var of Picard iteration}
\end{equation}

On the other hand, since $f$ is $Lip\left( \gamma \right) $ for $\gamma \geq
1$, by using $\left( \ref{inner iteration of solution of ode}\right) $, we
have, for any $\left[ s,t\right] \subset \left[ 0,T\right] $,%
\begin{eqnarray}
&&\left\Vert y\left( n+2\right) -y\left( n+1\right) \right\Vert _{1-var,
\left[ s,t\right] }  \label{inner inductive 1-var estimate} \\
&\leq &\left\vert f\right\vert _{Lip\left( \gamma \right) }\left\Vert
x\right\Vert _{1-var,\left[ s,t\right] }\left( \left\Vert y\left( n+1\right)
-y\left( n\right) \right\Vert _{1-var,\left[ s,t\right] }+\left\Vert y\left(
n+1\right) _{s}-y\left( n\right) _{s}\right\Vert \right) \text{.}  \notag
\end{eqnarray}%
Then we divide $\left[ 0,T\right] :=\cup _{j=0}^{m-1}\left[ t_{j},t_{j+1}%
\right] $ in such a way that%
\begin{equation*}
\left\vert f\right\vert _{Lip\left( \gamma \right) }\left\Vert x\right\Vert
_{1-var,\left[ t_{j},t_{j+1}\right] }\leq c<1\text{, }j=0,1,\dots ,m-1\text{.%
}
\end{equation*}%
Then, for $\left[ t_{j},t_{j+1}\right] $, $j=0,1,\dots ,m-1$, we let $n$
tends to infinity (in $\left( \ref{inner inductive 1-var estimate}\right) $%
), and get%
\begin{eqnarray*}
&&\overline{\lim }_{n\rightarrow \infty }\left\Vert y\left( n+1\right)
-y\left( n\right) \right\Vert _{1-var,\left[ t_{j},t_{j+1}\right] } \\
&\leq &c\,\overline{\lim }_{n\rightarrow \infty }\left\Vert y\left(
n+1\right) -y\left( n\right) \right\Vert _{1-var,\left[ t_{j},t_{j+1}\right]
}+c\,\overline{\lim }_{n\rightarrow \infty }\left\Vert y\left( n+1\right)
_{t_{j}}-y\left( n\right) _{t_{j}}\right\Vert \text{.}
\end{eqnarray*}%
Since $y\left( n\right) _{0}\equiv 0$, $\forall n\geq 0$, and $c\in \left(
0,1\right) $, we can prove inductively that%
\begin{equation*}
\lim_{n\rightarrow \infty }\left\Vert y\left( n+1\right) -y\left( n\right)
\right\Vert _{1-var,\left[ t_{j},t_{j+1}\right] }=0\text{, }j=0,1,\dots ,m-1%
\text{.}
\end{equation*}%
Thus%
\begin{equation*}
\lim_{n\rightarrow \infty }\left\Vert y\left( n+1\right) -y\left( n\right)
\right\Vert _{1-var,\left[ 0,T\right] }\leq
\sum_{j=0}^{m-1}\lim_{n\rightarrow \infty }\left\Vert y\left( n+1\right)
-y\left( n\right) \right\Vert _{1-var,\left[ t_{j},t_{j+1}\right] }=0\text{.}
\end{equation*}%
As a result, we have that $y\left( n\right) $ converge in $1$-variation as $n
$ tends to infinity (denote the limit as $\widetilde{y}$), and we have%
\begin{equation}
\lim_{n\rightarrow \infty }\max_{1\leq k\leq \left[ p\right] }\sup_{0\leq
s\leq t\leq T}\left\Vert \pi _{k}\left( S_{\left[ p\right] }\left( y\left(
n\right) \right) _{s,t}\right) -\pi _{k}\left( S_{\left[ p\right] }\left( 
\widetilde{y}\right) _{s,t}\right) \right\Vert =0\text{.}
\label{inner uniform convergence}
\end{equation}%
Based on $\left( \ref{inner iteration of solution of ode}\right) $ and let $n
$ tends to infinity, we have 
\begin{equation*}
\widetilde{y}_{t}=\int_{0}^{t}f\left( \widetilde{y}_{u}+\xi \right) dx_{u}%
\text{.}
\end{equation*}%
As a result, if denote $y$ as the unique solution of the ordinary
differential equation $\left( \ref{ODE in lemma}\right) $, then we have%
\begin{equation}
y=\widetilde{y}+\xi .  \label{inner relation between y and y tilde}
\end{equation}

Therefore, combine $\left( \ref{inner uniform bound of p-var of Picard
iteration}\right) $, $\left( \ref{inner uniform convergence}\right) $, $%
\left( \ref{inner relation between y and y tilde}\right) $ and use lower
semi-continuity of $p$-variation, we get, for interval $\left[ s,t\right] $
satisfying $\left\vert f\right\vert _{Lip\left( \gamma \right) }\left\Vert
S_{\left[ p\right] }\left( x\right) \right\Vert _{p-var,\left[ s,t\right]
}\leq 1$, 
\begin{eqnarray*}
\left\Vert S_{\left[ p\right] }\left( y\right) \right\Vert _{p-var,\left[ s,t%
\right] } &=&\left\Vert S_{\left[ p\right] }\left( \widetilde{y}+\xi \right)
\right\Vert _{p-var,\left[ s,t\right] }=\left\Vert S_{\left[ p\right]
}\left( \widetilde{y}\right) \right\Vert _{p-var,\left[ s,t\right] } \\
&\leq &\underline{\lim }_{n\rightarrow \infty }\left\Vert S_{\left[ p\right]
}\left( y\left( n\right) \right) \right\Vert _{p-var,\left[ s,t\right] }\leq
C_{p,\gamma }\left\vert f\right\vert _{Lip\left( \gamma \right) }\left\Vert
S_{\left[ p\right] }\left( x\right) \right\Vert _{p-var,\left[ s,t\right] }%
\text{.}
\end{eqnarray*}
\end{proof}

\bigskip

\noindent \textbf{Lemma \ref{Lemma important}} \ \ \textit{Suppose }$%
\mathcal{U}$\textit{\ and }$\mathcal{V}$\textit{\ are two Banach spaces, }$x:%
\left[ 0,T\right] \rightarrow \mathcal{V}$\textit{\ is a continuous bounded
variation path, }$f\in L\left( \mathcal{V},C^{\gamma }\left( \mathcal{U},%
\mathcal{U}\right) \right) $\textit{\ for }$\gamma >1$\textit{, and }$\xi
\in \mathcal{U}$\textit{. Denote }$y:\left[ 0,T\right] \rightarrow \mathcal{U%
}$\textit{\ as the\ unique solution to the ordinary differential equation}%
\begin{equation}
dy=f\left( y\right) dx\text{, \ }y_{0}=\xi \in \mathcal{U}\text{.}
\label{ODE in important Lemma}
\end{equation}%
\textit{Then for any }$p\in \lbrack 1,\gamma +1)$\textit{, there exists a
constant }$C_{p,\gamma }$\textit{, which only depends on }$p$\textit{\ and }$%
\gamma $\textit{, such that, for any }$0\leq s<t\leq T$\textit{, if we
denote }$y^{s,t}:\left[ 0,1\right] \rightarrow \mathcal{U}$\textit{\ as the
unique solution of the ordinary differential equation: (with }$y_{s}$\textit{%
\ denotes the value of }$y$\textit{\ in }$\left( \ref{ODE in important Lemma}%
\right) $\textit{\ at point }$s$\textit{) }%
\begin{eqnarray}
dy_{u}^{s,t} &=&\left( \sum_{k=1}^{\lfloor \gamma \rfloor }f^{\circ k}\pi
_{k}\left( \log _{\lfloor \gamma \rfloor +1}\left( S_{\lfloor \gamma \rfloor
+1}\left( x\right) _{s,t}\right) \right) \left( I_{d}\right) \left(
y_{u}^{s,t}\right) \right) du\text{, }u\in \left[ 0,1\right] \text{, }
\label{inner definition of y(s,t)} \\
y_{0}^{s,t} &=&y_{s}+f^{\circ \left( \lfloor \gamma \rfloor +1\right) }\pi
_{\lfloor \gamma \rfloor +1}\left( \log _{\lfloor \gamma \rfloor +1}\left(
S_{\lfloor \gamma \rfloor +1}\left( x\right) _{s,t}\right) \right) \left(
I_{d}\right) \left( y_{s}\right) \text{,}  \notag
\end{eqnarray}%
\textit{then }%
\begin{gather}
\left( 1\right) ,\left\Vert y_{t}-y_{1}^{s,t}\right\Vert \leq C_{p,\gamma
}\left\vert f\right\vert _{Lip\left( \gamma \right) }^{\gamma +1}\left\Vert
S_{\left[ p\right] }\left( x\right) \right\Vert _{p-var,\left[ s,t\right]
}^{\gamma +1}\text{,}  \label{inner estimate of ODE} \\
\left( 2\right) ,\left\Vert y_{t}-y_{s}-\sum_{k=1}^{\lfloor \gamma \rfloor
+1}f^{\circ k}\pi _{k}\left( S_{\lfloor \gamma \rfloor +1}\left( x\right)
_{s,t}\right) \left( I_{d}\right) \left( y_{s}\right) \right\Vert \leq
C_{p,\gamma }\left\vert f\right\vert _{Lip\left( \gamma \right) }^{\gamma
+1}\left\Vert S_{\left[ p\right] }\left( x\right) \right\Vert _{p-var,\left[
s,t\right] }^{\gamma +1}\text{.}  \notag
\end{gather}

\bigskip

\begin{proof}
\label{Proof of Lemma important}We only prove the first estimate in $\left( %
\ref{inner estimate of ODE}\right) $; the second follows based on Lemma \ref%
{Lemma Euler expansion of the solution of ODE} on page \pageref{Lemma Euler
expansion of the solution of ODE}. We prove the result when $\left\vert
f\right\vert _{Lip\left( \gamma \right) }=1$. The general case can be proved
by replacing $f$ by $\left\vert f\right\vert _{Lip\left( \gamma \right)
}^{-1}f$ and replacing $S_{\left[ p\right] }\left( x\right) $ by $\delta
_{\left\vert f\right\vert _{Lip\left( \gamma \right) }}\left( S_{\left[ p%
\right] }\left( x\right) \right) $ (in which case both $y$ in $\left( \ref%
{ODE in important Lemma}\right) $ and $y^{s,t}$ in $\left( \ref{inner
definition of y(s,t)}\right) $ would stay unchanged).

Denote $N:=\lfloor \gamma \rfloor +1$. Define $\omega :\left\{ \left(
s,t\right) |0\leq s\leq t\leq T\right\} \rightarrow \overline{%
\mathbb{R}
^{+}}$ as%
\begin{equation*}
\omega \left( s,t\right) :=\left\Vert S_{\left[ p\right] }\left( x\right)
\right\Vert _{p-var,\left[ s,t\right] }^{p}\text{.}
\end{equation*}%
Then it can be checked that, $\omega $ is continuous and is super-additive,
i.e. 
\begin{equation}
\omega \left( s,u\right) +\omega \left( u,t\right) \leq \omega \left(
s,t\right) \text{, \ }\forall 0\leq s\leq u\leq t\leq T\text{.}
\label{inner super additivity}
\end{equation}%
With $y^{s,t}$ defined at $\left( \ref{inner definition of y(s,t)}\right) $,
we define $\Gamma :\left\{ \left( s,t\right) |0\leq s\leq t\leq T\right\}
\rightarrow \mathcal{U}$ as%
\begin{equation*}
\Gamma _{s,t}:=y_{t}-y_{1}^{s,t}=y_{t}-y_{s}-\left( y_{1}^{s,t}-y_{s}\right) 
\text{.}
\end{equation*}%
For $0\leq s\leq u\leq t\leq T$, with $y^{s,u}$ defined at $\left( \ref%
{inner definition of y(s,t)}\right) $ and $x$ in $\left( \ref{ODE in
important Lemma}\right) $, we denote $\widetilde{y}^{u,t}$ as the unique
solution of the ordinary differential equation:%
\begin{eqnarray*}
d\widetilde{y}_{r}^{u,t} &=&\left( \sum_{k=1}^{N-1}f^{\circ k}\left(
I_{d}\right) \left( \widetilde{y}_{r}^{u,t}\right) \pi _{k}\left( \log
_{N}\left( S_{N}\left( x\right) _{u,t}\right) \right) \right) dr\text{, }%
r\in \left[ 0,1\right] \text{, } \\
\widetilde{y}_{0}^{u,t} &=&y_{1}^{s,u}+f^{\circ N}\left( I_{d}\right) \left(
y_{1}^{s,u}\right) \pi _{N}\left( \log _{N}\left( S_{N}\left( x\right)
_{u,t}\right) \right) \text{.}
\end{eqnarray*}%
For $0\leq s\leq u\leq t\leq T$, we denote piecewise continuous path $%
y^{s,u,t}:\left[ 0,2\right] \rightarrow \mathcal{U}$ by assigning%
\begin{equation}
y_{r}^{s,u,t}:=y_{r}^{s,u}\text{ when }r\in \left[ 0,1\right] \text{ and }%
y_{r}^{s,u,t}:=\widetilde{y}_{r-1}^{u,t}\text{ when }r\in (1,2]\text{.}
\label{inner definitiono f y(s,u,t)}
\end{equation}

Firstly, suppose $\left[ s,t\right] $ is an interval satisfying $\omega
\left( s,t\right) \leq 1$ and $u\in \left( s,t\right) $. Then%
\begin{eqnarray*}
\left\Vert \Gamma _{s,u}+\Gamma _{u,t}-\Gamma _{s,t}\right\Vert
&=&\left\Vert y_{1}^{s,u}-y_{s}+y_{1}^{u,t}-y_{u}-\left(
y_{1}^{s,t}-y_{s}\right) \right\Vert \\
&\leq &\left\Vert y_{2}^{s,u,t}-y_{1}^{s,t}\right\Vert +\left\Vert \left(
y_{2}^{s,u,t}-y_{1}^{s,u}\right) -\left( y_{1}^{u,t}-y_{u}\right)
\right\Vert \text{.}
\end{eqnarray*}%
Then, based on Lemma \ref{Lemma difference of two steps and one step} and
Lemma \ref{Lemma Euler expansion of the solution of ODE}, we have ($%
\left\vert f\right\vert _{Lip\left( \gamma \right) }=1$ and $\omega \left(
s,t\right) \leq 1$, $\left\{ \gamma \right\} :=\gamma -\lfloor \gamma
\rfloor $)%
\begin{eqnarray}
&&\left\Vert \Gamma _{s,u}+\Gamma _{u,t}-\Gamma _{s,t}\right\Vert
\label{inner divide an interval into two} \\
&\leq &C_{\gamma }\omega \left( s,t\right) ^{\frac{\gamma +1}{p}}+\left\Vert
\sum_{k=1}^{N}f^{\circ k}\pi _{k}\left( S_{\left[ p\right] }\left( x\right)
_{u,t}\right) \left( \left( I_{d}\right) \left( y_{1}^{s,u}\right) -\left(
I_{d}\right) \left( y_{u}\right) \right) \right\Vert  \notag \\
&\leq &C_{\gamma }\omega \left( s,t\right) ^{\frac{\gamma +1}{p}}+C_{\gamma
}\omega \left( s,t\right) ^{\frac{1}{p}}\left\Vert \Gamma _{s,u}\right\Vert
+C_{\gamma }\left\Vert y_{u}-y_{s}-\left( y_{1}^{s,u}-y_{s}\right)
\right\Vert ^{\left\{ \gamma \right\} }\omega \left( s,t\right) ^{\frac{N}{p}%
}\text{.}  \notag
\end{eqnarray}%
Based on the definition of $y^{s,u}$ (at $\left( \ref{inner definition of
y(s,t)}\right) $), when $\omega \left( s,t\right) \leq 1$, we have%
\begin{equation}
\left\Vert y_{1}^{s,u}-y_{s}\right\Vert \leq \left\Vert y^{s,u}\right\Vert
_{1-var,\left[ 0,1\right] }+\left\Vert y_{0}^{s,u}-y_{s}\right\Vert \leq
C_{\gamma }\omega \left( s,t\right) ^{\frac{1}{p}}\text{.}
\label{inner estimate of y^(s,u)}
\end{equation}%
On the other hand, according to Lemma \ref{Lemma uniform bound in p-var of
solution of ode}, there exists constant $C_{p,\gamma }$ (which only depends
on $p$ and $\gamma $, and is finite whenever $\gamma >p-1$), such that for
any interval $\left[ s,t\right] $ satisfying $\omega \left( s,t\right) \leq
1 $, we have%
\begin{equation}
\left\Vert y_{u}-y_{s}\right\Vert \leq C_{p,\gamma }\omega \left( s,t\right)
^{\frac{1}{p}}\text{.}  \label{inner estimation of ode solution}
\end{equation}%
As a result, combining $\left( \ref{inner estimate of y^(s,u)}\right) $ and $%
\left( \ref{inner estimation of ode solution}\right) $, we have, when $%
\omega \left( s,t\right) \leq 1$,%
\begin{equation*}
\left\Vert \left( y_{u}-y_{s}\right) -\left( y_{1}^{s,u}-y_{s}\right)
\right\Vert ^{\left\{ \gamma \right\} }\omega \left( s,t\right) ^{\frac{N}{p}%
}\leq C_{p,\gamma }\omega \left( s,t\right) ^{\frac{\gamma +1}{p}}\text{.}
\end{equation*}%
Then, continuing with $\left( \ref{inner divide an interval into two}\right) 
$, we get, for any interval $\left[ s,t\right] $ satisfying $\omega \left(
s,t\right) \leq 1$ and any $u\in \left( s,t\right) $,%
\begin{equation}
\left\Vert \Gamma _{s,t}\right\Vert \leq \left( 1+C_{\gamma }\omega \left(
s,t\right) ^{\frac{1}{p}}\right) \left( \left\Vert \Gamma _{s,u}\right\Vert
+\left\Vert \Gamma _{u,t}\right\Vert \right) +C_{p,\gamma }\omega \left(
s,t\right) ^{\frac{\gamma +1}{p}}\text{.}  \label{inner estimation of Gamma}
\end{equation}

With $C_{\gamma }$ and $C_{p,\gamma }$ in $\left( \ref{inner estimation of
Gamma}\right) $, suppose $\left[ s,t\right] $ is an interval satisfying $%
\omega \left( s,t\right) \leq 1$, denote 
\begin{equation*}
\delta :=\left( C_{\gamma }^{p}\vee C_{p,\gamma }^{\frac{p}{\gamma +1}%
}\right) \omega \left( s,t\right) .
\end{equation*}%
Then since $\omega $ is super-additive (i.e. $\left( \ref{inner super
additivity}\right) $), by setting $[t_{0}^{0},t_{1}^{0})=[s,t)$ and
recursively dividing $[t_{j}^{n},t_{j+1}^{n})=[t_{2j}^{n+1},t_{2j+1}^{n+1})%
\sqcup \lbrack t_{2j+1}^{n+1},t_{2j+2}^{n+1})$ in such a way that 
\begin{equation*}
\omega \left( t_{2j}^{n+1},t_{2j+1}^{n+1}\right) =\omega \left(
t_{2j+1}^{n+1},t_{2j+2}^{n+1}\right) \leq \frac{1}{2}\omega \left(
t_{j}^{n},t_{j+1}^{n}\right) \text{, }j=0,1,\dots ,2^{n}-1\text{, }n\geq 0%
\text{,}
\end{equation*}%
we have, based on $\left( \ref{inner estimation of Gamma}\right) $, 
\begin{eqnarray}
\left\Vert \Gamma _{s,t}\right\Vert &\leq &\overline{\lim }_{n\rightarrow
\infty }\left( \sum_{k=0}^{n}\left( \dprod\limits_{j=0}^{k}\left( 1+2^{-%
\frac{j}{p}}\delta ^{\frac{1}{p}}\right) \right) \left( \frac{1}{2}\right)
^{\left( \frac{\gamma +1}{p}-1\right) k}\right) \delta ^{\frac{\gamma +1}{p}}
\label{inner estimate of Gamma} \\
&&+\overline{\lim }_{n\rightarrow \infty }\left(
\dprod\limits_{j=0}^{n}\left( 1+2^{-\frac{j}{p}}\delta ^{\frac{1}{p}}\right)
\right) \left( \sum_{j=0}^{2^{n}-1}\left\Vert \Gamma
_{t_{j}^{n},t_{j+1}^{n}}\right\Vert \right) \text{.}  \notag \\
&\leq &\exp \left( \frac{2^{\frac{1}{p}}}{2^{\frac{1}{p}}-1}\delta ^{\frac{1%
}{p}}\right) \left( \frac{2^{\frac{\gamma +1}{p}-1}}{2^{\frac{\gamma +1}{p}%
-1}-1}\delta ^{\frac{\gamma +1}{p}}+\overline{\lim }_{n\rightarrow \infty
}\sum_{j=0}^{2^{n}-1}\left\Vert \Gamma _{t_{j}^{n},t_{j+1}^{n}}\right\Vert
\right) \text{.}  \notag
\end{eqnarray}

Then we prove that $\overline{\lim }_{n\rightarrow \infty
}\sum_{j=0}^{2^{n}-1}\left\Vert \Gamma _{t_{j}^{n},t_{j+1}^{n}}\right\Vert
=0 $. Since $x:\left[ 0,T\right] \rightarrow \mathcal{V}$ is a continuous
bounded variation path and $y:\left[ 0,T\right] \rightarrow \mathcal{U}$ is
the solution of the ordinary differential equation%
\begin{equation*}
dy=f\left( y\right) dx\text{, \ }y_{0}=\xi \text{,}
\end{equation*}%
we have that ($\left\vert f\right\vert _{Lip\left( \gamma \right) }=1$, $%
\gamma >1$)%
\begin{equation*}
\left\Vert y\right\Vert _{1-var,\left[ s,t\right] }\leq \left\Vert
x\right\Vert _{1-var,\left[ s,t\right] }\text{, }\forall 0\leq s\leq t\leq T%
\text{.}
\end{equation*}%
Thus,%
\begin{eqnarray}
&&\left\Vert y_{t}-y_{s}-\sum_{k=1}^{N}f^{\circ k}\pi _{k}\left( S_{\left[ p%
\right] }\left( x\right) _{s,t}\right) \left( I_{d}\right) \left(
y_{s}\right) \right\Vert  \label{inner estimation of ODE1} \\
&=&\left\Vert \idotsint\nolimits_{s<u_{1}<\cdots <u_{N}<t}f^{\circ N}\left(
dx_{u_{1}}\otimes \cdots \otimes dx_{u_{N}}\right) \left( \left(
I_{d}\right) \left( y_{u_{1}}\right) -\left( I_{d}\right) \left(
y_{s}\right) \right) \right\Vert  \notag \\
&\leq &C_{\gamma }\idotsint\nolimits_{s<u_{1}<\cdots <u_{N}<t}\left\Vert
y_{u_{1}}-y_{s}\right\Vert ^{\left\{ \gamma \right\} }\left\Vert
dx_{u_{1}}\right\Vert \cdots \left\Vert dx_{u_{N}}\right\Vert  \notag \\
&\leq &C_{\gamma }\sup_{u\in \left[ s,t\right] }\left\Vert
y_{u}-y_{s}\right\Vert ^{\left\{ \gamma \right\} }\left\Vert x\right\Vert
_{1-var,\left[ s,t\right] }^{N}\leq C_{\gamma }\left\Vert x\right\Vert
_{1-var,\left[ s,t\right] }^{\gamma +1}\text{.}  \notag
\end{eqnarray}%
On the other hand, based on Lemma \ref{Lemma Euler expansion of the solution
of ODE}, we have%
\begin{equation}
\left\Vert y_{1}^{s,t}-y_{s}-\sum_{k=1}^{N}f^{\circ k}\pi _{k}\left( S_{ 
\left[ p\right] }\left( x\right) _{s,t}\right) \left( I_{d}\right) \left(
y_{s}\right) \right\Vert \leq C_{\gamma }\left\Vert S_{\left[ p\right]
}\left( x\right) _{s,t}\right\Vert ^{\gamma +1}\leq C_{p,\gamma }\left\Vert
x\right\Vert _{1-var,\left[ s,t\right] }^{\gamma +1}\text{.}
\label{inner estimation of ODE2}
\end{equation}%
Thus, combining $\left( \ref{inner estimation of ODE1}\right) $ and $\left( %
\ref{inner estimation of ODE2}\right) $, we get%
\begin{equation*}
\left\Vert \Gamma _{t_{j}^{n},t_{j+1}^{n}}\right\Vert \leq C_{p,\gamma
}\left\Vert x\right\Vert _{1-var,\left[ t_{j}^{n},t_{j+1}^{n}\right]
}^{\gamma +1}\text{, }j=0,1,\dots ,2^{n}-1\text{, }n\geq 0\text{,}
\end{equation*}%
and~we have ($\gamma \geq 1$)%
\begin{equation*}
\overline{\lim }_{n\rightarrow \infty }\sum_{j=0}^{2^{n}-1}\left\Vert \Gamma
_{t_{j}^{n},t_{j+1}^{n}}\right\Vert =0\text{.}
\end{equation*}

Thus, continuing with $\left( \ref{inner estimate of Gamma}\right) $, we get
that, there exists constant $C_{p,\gamma }$, which only depends on $p$ and $%
\gamma $, and is finite whenever $\gamma >p-1$, such that, for any interval $%
\left[ s,t\right] $ satisfying $\omega \left( s,t\right) \leq 1$, we have%
\begin{equation}
\left\Vert y_{t}-y_{1}^{s,t}\right\Vert =\left\Vert \Gamma _{s,t}\right\Vert
\leq C_{p,\gamma }\omega \left( s,t\right) ^{\frac{\gamma +1}{p}}\text{.}
\label{inner result}
\end{equation}

For $\left[ s,t\right] $ satisfying $\omega \left( s,t\right) >1$, as in
Prop 5.10 \cite{Friz and Victoir book}, we decompose $[s,t)=\sqcup
_{j=0}^{n-1}[t_{j},t_{j+1})$ in such a way that $\omega \left(
t_{j},t_{j+1}\right) =1$, $j=0,1,\dots ,n-2$, and $\omega \left(
t_{n-1},t_{n}\right) \leq 1$. Then by using super-additivity of $\omega $,
we have $n-1\leq \omega \left( s,t\right) $, and%
\begin{eqnarray*}
\left\Vert y_{t}-y_{s}\right\Vert  &\leq &\sum_{j=0}^{n-1}\left\Vert
y_{t_{j+1}}-y_{t_{j}}\right\Vert \leq C_{p,\gamma }\left( n-1+\omega \left(
t_{n-1},t_{n}\right) ^{\frac{1}{p}}\right)  \\
&\leq &C_{p,\gamma }n\leq C_{p,\gamma }\left( \omega \left( s,t\right)
+1\right) \leq 2C_{p,\gamma }\omega \left( s,t\right) \text{.}
\end{eqnarray*}%
On the other hand, when $\omega \left( s,t\right) \geq 1$, 
\begin{eqnarray*}
\left\Vert y_{1}^{s,t}-y_{s}\right\Vert  &\leq &\left\Vert
y_{1}^{s,t}-y_{0}^{s,t}\right\Vert +\left\Vert y_{0}^{s,t}-y_{s}\right\Vert 
\\
&\leq &C_{\gamma }\left\Vert S_{\left[ p\right] }\left( x\right)
_{s,t}\right\Vert ^{N-1}+C_{\gamma }\left\Vert S_{\left[ p\right] }\left(
x\right) _{s,t}\right\Vert ^{N}\leq C_{\gamma }\omega \left( s,t\right) ^{%
\frac{N}{p}}\text{.}
\end{eqnarray*}%
Therefore, when $\omega \left( s,t\right) \geq 1$,%
\begin{equation*}
\left\Vert y_{t}-y_{1}^{s,t}\right\Vert =\left\Vert y_{t}-y_{s}-\left(
y_{1}^{s,t}-y_{s}\right) \right\Vert \leq C_{p,\gamma }\omega \left(
s,t\right) +C_{\gamma }\omega \left( s,t\right) ^{\frac{N}{p}}\leq
C_{p,\gamma }\omega \left( s,t\right) ^{\frac{\gamma +1}{p}}\text{.}
\end{equation*}
\end{proof}

\begin{lemma}
\label{Lemma convergence of solution of ode}Suppose $f\in L\left( \mathcal{V}%
,C^{\gamma }\left( \mathcal{U},\mathcal{U}\right) \right) $ for some $\gamma
>1$. For $g\in G^{\lfloor \gamma \rfloor +1}\left( \mathcal{V}\right) $ and $%
\xi \in \mathcal{U}$, define $y\left( g,\xi \right) :\left[ 0,1\right]
\rightarrow \mathcal{U}$ as the unique solution of the ordinary differential
equation: 
\begin{eqnarray*}
dy_{u} &=&\sum_{k=1}^{\lfloor \gamma \rfloor }f^{\circ k}\pi _{k}\left( \log
_{\lfloor \gamma \rfloor +1}\left( g\right) \right) \left( I_{d}\right)
\left( y_{u}\right) du\text{, \ }u\in \left[ 0,1\right] \text{, } \\
y_{0} &=&\xi +f^{\circ \left( \lfloor \gamma \rfloor +1\right) }\pi
_{\lfloor \gamma \rfloor +1}\left( \log _{\lfloor \gamma \rfloor +1}\left(
g\right) \right) \left( I_{d}\right) \left( \xi \right) \in \mathcal{U}\text{%
,}
\end{eqnarray*}%
If there exist $\left\{ g^{l}\right\} _{l\geq 1}\subset G^{\lfloor \gamma
\rfloor +1}\left( \mathcal{V}\right) $ and $\left\{ \xi ^{l}\right\} _{\geq
1}\subset \mathcal{U}$ such that 
\begin{equation*}
\lim_{l\rightarrow \infty }\max_{1\leq k\leq \lfloor \gamma \rfloor
+1}\left\Vert \pi _{k}\left( g^{l}\right) -\pi _{k}\left( g\right)
\right\Vert =0\text{ and }\lim_{l\rightarrow \infty }\left\Vert \xi ^{l}-\xi
\right\Vert =0\text{,}
\end{equation*}%
then%
\begin{equation*}
\lim_{l\rightarrow \infty }\sup_{t\in \left[ 0,1\right] }\left\Vert y\left(
g^{l},\xi ^{l}\right) _{t}-y\left( g,\xi \right) _{t}\right\Vert =0\text{.}
\end{equation*}
\end{lemma}

\begin{proof}
Since $\sum_{k=1}^{\lfloor \gamma \rfloor }f^{\circ k}\left( I_{d}\right)
\in C^{1}\left( \mathcal{U},\mathcal{U}\right) $, based on Thm 3.15 \cite%
{Friz and Victoir book} (their result extends naturally to ordinary
differential equations in Banach spaces), we get 
\begin{eqnarray*}
&&\sup_{t\in \left[ 0,1\right] }\left\Vert y\left( g^{l},\xi ^{l}\right)
_{t}-y\left( g,\xi \right) _{t}\right\Vert \\
&\leq &C_{f}\left( \left\Vert y\left( g^{l},\xi ^{l}\right) _{0}-y\left(
g,\xi \right) _{0}\right\Vert +\sum_{k=1}^{\lfloor \gamma \rfloor
}\left\Vert f^{\circ k}\left( \pi _{k}\left( \log _{\lfloor \gamma \rfloor
+1}\left( g^{l}\right) \right) -\pi _{k}\left( \log _{\lfloor \gamma \rfloor
+1}\left( g\right) \right) \right) \left( I_{d}\right) \right\Vert _{\infty
}\right) \\
&\leq &C_{f}\left( \left\Vert \xi ^{l}-\xi \right\Vert +\left\Vert \pi
_{\lfloor \gamma \rfloor +1}\left( \log _{\lfloor \gamma \rfloor +1}\left(
g\right) \right) \right\Vert \left\Vert \xi ^{l}-\xi \right\Vert ^{\left\{
\gamma \right\} }+\sum_{k=1}^{\lfloor \gamma \rfloor +1}\left\Vert \pi
_{k}\left( \log _{\lfloor \gamma \rfloor +1}\left( g^{l}\right) \right) -\pi
_{k}\left( \log _{\lfloor \gamma \rfloor +1}\left( g\right) \right)
\right\Vert \right) \text{.}
\end{eqnarray*}
\end{proof}

\begin{proof}[Proof of Theorem \protect\ref{Theorem rough Euler expansion of
solution of RDE}]
\label{Proof of Theorem rough Euler expansion of solution of RDE}Based on
the definition of geometric $p$-rough path (in Definition \ref{Definition of
geometric rough path} on page \pageref{Definition of geometric rough path}),
there exists a sequence of continuous bounded variation paths $\left\{
x^{l}\right\} :\left[ 0,T\right] \rightarrow \mathcal{V}$, such that%
\begin{equation*}
\lim_{l\rightarrow \infty }d_{p}\left( S_{\left[ p\right] }\left(
x^{l}\right) ,X\right) =0\text{.}
\end{equation*}%
As a result, we have%
\begin{equation}
\ \lim_{l\rightarrow \infty }\left\Vert S_{\left[ p\right] }\left(
x^{l}\right) \right\Vert _{p-var,\left[ s,t\right] }=\left\Vert X\right\Vert
_{p-var,\left[ s,t\right] }\text{, \ }\forall 0\leq s\leq t\leq T\text{,}
\label{inner condition 11}
\end{equation}%
and (based on Thm 3.1.3 \cite{LyonsQian}) 
\begin{equation}
\lim_{l\rightarrow \infty }\max_{n=1,2,\dots ,\lfloor \gamma \rfloor
+1}\left\Vert \pi _{n}\left( S_{\lfloor \gamma \rfloor +1}\left(
x^{l}\right) _{s,t}\right) -\pi _{n}\left( S_{\lfloor \gamma \rfloor
+1}\left( X\right) _{s,t}\right) \right\Vert =0\text{, }\forall 0\leq s\leq
t\leq T\text{.}  \label{inner condition 1}
\end{equation}

On the other hand, denote $y^{l}:\left[ 0,T\right] \rightarrow \mathcal{U}$
as the unique solution of the ordinary differential equation 
\begin{equation*}
dy^{l}=f\left( y^{l}\right) dx^{l}\text{, }y_{0}^{l}=\xi \text{,}
\end{equation*}%
and denote $Y:=\pi _{G^{\left[ p\right] }\left( \mathcal{U}\right) }\left(
Z\right) $ with $Z$ denotes the unique solution (in the sense of Definition %
\ref{Definition of solution of RDE}) of the rough differential equation%
\begin{equation*}
dY=f\left( Y\right) dX\text{, }Y_{0}=\xi \text{.}
\end{equation*}%
Then based on universal limit theorem (Thm 5.3 \cite{Lyonsnotes}), we have 
\begin{equation}
\lim_{l\rightarrow \infty }\left\Vert y_{t}^{l}-\pi _{1}\left( Y_{t}\right)
\right\Vert =0\text{, \ }\forall t\in \left[ 0,T\right] \text{.}
\label{inner condition 2}
\end{equation}

For $0\leq s\leq t\leq T$ and $l\geq 1$, denote $y^{s,t,l}:\left[ 0,1\right]
\rightarrow \mathcal{U}$ as the unique solution of the ordinary differential
equation: 
\begin{eqnarray*}
dy_{u}^{s,t,l} &=&\sum_{k=1}^{\lfloor \gamma \rfloor }f^{\circ k}\pi
_{k}\left( \log _{\lfloor \gamma \rfloor +1}\left( S_{\lfloor \gamma \rfloor
+1}\left( x^{l}\right) \right) \right) \left( I_{d}\right) \left(
y_{u}^{s,t,l}\right) du\text{, \ }u\in \left[ 0,1\right] \text{, } \\
y_{0}^{s,t,l} &=&y_{s}^{l}+f^{\circ \left( \lfloor \gamma \rfloor +1\right)
}\pi _{\lfloor \gamma \rfloor +1}\left( \log _{\lfloor \gamma \rfloor
+1}\left( S_{\lfloor \gamma \rfloor +1}\left( x^{l}\right) \right) \right)
\left( I_{d}\right) \left( y_{s}^{l}\right) \in \mathcal{U}\text{.}
\end{eqnarray*}%
For $0\leq s\leq t\leq T$, denote $y^{s,t}:\left[ 0,1\right] \rightarrow 
\mathcal{U}$ as the unique solution of the ordinary differential equation:%
\begin{eqnarray*}
dy_{u}^{s,t} &=&\sum_{k=1}^{\lfloor \gamma \rfloor }f^{\circ k}\pi
_{k}\left( \log _{\lfloor \gamma \rfloor +1}\left( S_{\lfloor \gamma \rfloor
+1}\left( X\right) \right) \right) \left( I_{d}\right) \left(
y_{u}^{s,t}\right) du\text{, \ }u\in \left[ 0,1\right] \text{, } \\
y_{0}^{s,t} &=&\pi _{1}\left( Y_{s}\right) +f^{\circ \left( \lfloor \gamma
\rfloor +1\right) }\pi _{\lfloor \gamma \rfloor +1}\left( \log _{\lfloor
\gamma \rfloor +1}\left( S_{\lfloor \gamma \rfloor +1}\left( X\right)
\right) \right) \left( I_{d}\right) \left( \pi _{1}\left( Y_{s}\right)
\right) \in \mathcal{U}\text{.}
\end{eqnarray*}%
Then according to Lemma \ref{Lemma convergence of solution of ode}, we have%
\begin{equation}
\lim_{l\rightarrow \infty }\left\Vert y_{1}^{s,t,l}-y_{1}^{s,t}\right\Vert =0%
\text{.}  \label{inner condition 3}
\end{equation}

Based on Lemma \ref{Lemma important}, for each $l\geq 1$, we have%
\begin{gather}
\left\Vert y_{t}^{l}-y_{1}^{s,t,l}\right\Vert \leq C_{p,\gamma }\left(
\left\vert f\right\vert _{Lip\left( \gamma \right) }\left\Vert S_{\left[ p%
\right] }\left( x^{l}\right) \right\Vert _{p-var,\left[ s,t\right] }\right)
^{\gamma +1}\text{,}  \label{inner repeat result of ode} \\
\left\Vert y_{t}^{l}-y_{s}^{l}-\sum_{k=1}^{\lfloor \gamma \rfloor
+1}f^{\circ k}\pi _{k}\left( S_{\lfloor \gamma \rfloor +1}\left(
x^{l}\right) _{s,t}\right) \left( I_{d}\right) \left( y_{s}^{l}\right)
\right\Vert \leq C_{p,\gamma }\left( \left\vert f\right\vert _{Lip\left(
\gamma \right) }\left\Vert S_{\left[ p\right] }\left( x^{l}\right)
\right\Vert _{p-var,\left[ s,t\right] }\right) ^{\gamma +1}\text{.}  \notag
\end{gather}%
Combining $\left( \ref{inner condition 11}\right) $, $\left( \ref{inner
condition 1}\right) $, $\left( \ref{inner condition 2}\right) $ and $\left( %
\ref{inner condition 3}\right) $, we let $l\rightarrow \infty $ in $\left( %
\ref{inner repeat result of ode}\right) $, and get Theorem \ref{Theorem
rough Euler expansion of solution of RDE}.
\end{proof}

\section{Acknowledgement}

The research of Youness Boutaib and Lajos Gergely Gyurk\'{o} are supported
by Oxford-Man Institute. The research of Terry Lyons and Danyu Yang are
supported by European Research Council under the European Union's Seventh
Framework Programme (FP7-IDEAS-ERC)/ ERC grant agreement nr. 291244. The
research of Terry Lyons is supported by EPSRC grant EP/H000100/1. The
authors are grateful for the support of Oxford-Man Institute.

\end{document}